\documentclass[reqno]{amsart}%
\usepackage{amssymb}
\usepackage{graphicx}
\usepackage{amsthm}
\usepackage{amsfonts}
\usepackage{latexsym}
\usepackage{amsmath}
\usepackage{setspace}
\usepackage{enumerate}   %
\providecommand{\U}[1]{\protect\rule{.1in}{.1in}}
\newtheorem{theorem}{Theorem}[section]

\newtheorem{corollary}[theorem]{Corollary}
\newtheorem{lemma}[theorem]{Lemma}

\newtheorem{remark}[theorem]{Remark}

\newcommand{\eps}{\varepsilon}

\DeclareMathOperator{\dt}{dt}

\newcommand{\R}{\mathbb{R}}
\newcommand{\N}{\mathbb{N}}
\newcommand{\cA}{\mathcal A}
\newcommand{\cC}{\mathcal C}
\newcommand{\cD}{\mathcal D}

\DeclareMathOperator{\e}{e}
\newcommand{\Id}{\textrm{Id}}
\DeclareMathOperator{\im}{Im}

\author[C. M. Silva]{C\'esar M. Silva}
\email{csilva@ubi.pt}
\address{
Universidade da Beira Interior
\\
Rua Marqu\^es d'\'Avila e Bolama,
6201-001 Covilh\~a
Portugal.\\
}

\begin{document}
\title[Admissibility and generalized nonuniform dichotomies for discrete dynamics]{Admissibility and generalized nonuniform dichotomies for discrete dynamics}

\begin{abstract}
We obtain characterizations of nonuniform dichotomies, defined by general growth rates, based on admissibility conditions. Additionally, we use the obtained characterizations to derive robustness results for the considered dichotomies. As particular cases, we recover several results in the literature concerning nonuniform exponential dichotomies and nonuniform polynomial dichotomies as well as new results for nonuniform dichotomies with logarithmic growth.
\end{abstract}

\thanks{C. M. Silva was partially supported by FCT through CMUBI (project UIDB/MAT/00212/2020).}
\maketitle

\section{Introduction}

Let $X$ be a Banach space and $A(t):X \to X$ be a family of linear bounded operators varying continuously with $t \in \R_0^+$. In the finite dimensional setting ($X=\R^n$), Perron observed a relation between exponential stability of the solutions of the non-autonomous linear evolution equation $\frac{du}{dt}=A(t)u$ and the existence of bounded solutions of the equation
\begin{equation}\label{eq:intro1}
\frac{du}{dt}=A(t)u+f(t)
\end{equation}
for bounded continuous perturbations $f:\R_0^+\to \R$. The works of Perron on the stability of linear differential equations under sufficiently small perturbations~\cite{Pe.1930} can be considered the starting point in the study of admissibility.

More generally, the condition that for each fixed $f$, belonging to some suitable space $Y$, the equation~\eqref{eq:intro1} has a unique solution in some space $D$ is known as \emph{admissibility} of the pair of spaces where we take the perturbations, $Y$, and where we search for solutions, $D$. Note that, considering the operator $\mathcal L:D \to Y$ given by $(\mathcal L u)(t)=\frac{du}{dt}-A(t)u$, the admissibility of the pair of spaces $Y$ and $D$ can be rephrased in the context of spectral properties of $\mathcal L$.

The notion of (uniform) exponential dichotomy is a natural generalization of the notion of exponential stability that, instead of stability in the whole space, requires that there is a splitting of the phase space into two complementary directions, the stable and unstable directions, where we have, respectively, stable and unstable behaviour of solutions.
Several results have been obtained relating the existence of exponential dichotomies and admissibility. The books by Massera and Sch\"affer~\cite{Ma.Sch}, {Dalecki\v\i} and Kre\v\i n~\cite{Dal} and Coppel~\cite{Co} constitute an important source of early contributions to the theory of admissibility.

Despite its importance in the theory, the notion of (uniform) exponential dichotomy is sometimes too restrictive and therefore it is important to look for more general hyperbolic behavior. We can generalize the concept of nonuniform dichotomy by allowing some loss of hyperbolicity along the trajectories, leading to notions similar to Pesin's nonuniform hyperbolicity~\cite{Pe.1976,Pe.1977-1,Pe.1977-2}. In another direction, we can assume that the asymptotic behavior is not exponential. This reasoning leads to uniform dichotomies with asymptotic behavior given by general growth rates, a notion considered in the work of Naulin and Pinto in \cite{Na.Pi.1995,Pi.1994}. We refer the recent book~\cite{B.D.V.AH.2018} for a recent account of the relation between admissibility and hyperbolicity, including the case of nonuniform exponential dichotomies, a type of dichotomic behavior where some exponential loss of hyperbolicity along the trajectories is allowed (see for example the work of Barreira and Valls~\cite{Ba.Va.2008-1} and papers~\cite{Lu.Me.2014,Pr.Pr.Cr.2012,Sa.Ba.Sa.2013}). Admissibility was used to establish the existence of nonuniform discrete exponential dichotomies and, more generally, of exponential dichotomies with respect to a sequence of norms in the interesting paper~\cite{BDV.admis} by Barreira, Dragi\v cevi\'c and Valls.

Going further on the direction of more general dichotomic behavior, we can consider nonuniform dichotomies that do not necessarily have exponential growth. Among the long list of works, several results were obtained in \cite{Ba.Me.Po,Ba.Ch.Va.2013,Ba.Va.2008-2,Be.Si.,Be.Si.2013,Ch} for general nonuniform behavior and in \cite{Ba.Va.2009,Be.Si.2009,Be.Si.2012} for nonuniform polynomial behavior. In~\cite{D.MN.2019}, Dragicevic used the notion of admissibility to establish the existence of nonuniform polynomial dichotomies in the discrete time setting. Our work was motivated by~\cite{D.MN.2019,BDV.admis}.

In this work we consider the discrete time case. Our main objective is to generalize the results on existence of nonuniform exponential dichotomies~\cite{BDV.admis} and of nonuniform polynomial dichotomies~\cite{D.MN.2019}, based on admissibility of suitable pairs of spaces, to the more general case of nonuniform $\mu$-dichotomy. More precisely, we start by defining the notion of nonuniform $\mu$-dichotomy with respect to a sequence of norms, obtain results for these concept of dichotomy and use those results to address the case of $\mu$-dichotomies and strong $\mu$-dichotomies (a path already followed in~\cite{D.MN.2019} in the case of polynomial dichotomies). To the best of our knowledge the concept of nonuniform $\mu$-dichotomy with respect to a sequence of norms was introduced here for the first time, even though it is a natural generalization of the corresponding exponential and polynomial concepts.

We also want to highlight three main aspects of our work. In the first place, we work directly with discrete growth rates: we assume the growth rates are given by a sequence $(\rho_n)_{n \in \N}$ instead of being given by $(g(n))_{n \in \N}$, where the function $g:\R_0^+ \to \R_0^+$ is usually assumed to be increasing and differentiable. This allows our assumptions to be on properties of the sequence $(\rho_n)_{n \in \N}$ rather than properties of the function $g$. This option requires a technical lemma, relying on a result concerning the smooth global approximation of a Lipschitz function by differentiable functions, that will be used to obtain some important estimates in the proofs of our main results. Another aspect we want to refer is that it is not straightforward to guess, from the exponential and polynomial cases, how to define the operator whose invertibility relates to the existence of $\mu$-dichotomies (see~\eqref{eq:def-TmuZ}). We also note that we obtain, as corollaries of our theorems, not only the already studied exponential and polynomial cases but also the case of logarithmic growth rates. Finally, we obtain robustness results as a consequence of our characterization of $\mu$-dichotomies, and recover some existent results in the literature as particular cases of our results.

\section{Generalized dichotomies}

We say that a sequence $(\mu_m)_{m \in \N}$ is a \emph{discrete growth rate} if it is strictly increasing and $\lim_{n \to +\infty} \mu_m=+\infty$.

Let $(X, \|\cdot\|)$ be a Banach space and $(A_n)_{n \in\N}$ be a sequence of bounded linear operators acting on $X$. Define $\Delta=\{(m,n)\in \N \times \N: m\ge n\}$ and denote the \emph{discrete evolution family} associated with the sequence $(A_n)_{n \in \N}$ by $\cA=(\cA_{m,n})_{(m,n)\in\Delta}$, where
\[
\cA_{m,n}=
\begin{cases}
A_{m-1}\cdots A_n, & m>n\\
Id, & m=n
\end{cases}.
\]

Fix a growth rate $(\mu_n)_{n \in \N}$, let $(A_m)_{m\in \N}$ be a sequence of linear operators and let $(\|\cdot\|_m)_{m \in \N}$ be a sequence of norms in $X$ such that, for each fixed $m$, the norm $\|\cdot\|_m$ is equivalent to $\|\cdot\|$. We say that the sequence of linear operators $(A_m)_{m\in \N}$ (or alternatively that the linear difference equation $x_{m+1} = A_m x_m, \ m \in \N$) admits a \textit{$\mu$-dichotomy with respect to the sequence of norms} if there are projections $P_m$, $m\in\N$, such that $A_m|\ker P_m \to \ker P_{m+1}$ is invertible,
   $$P_m \cA_{m,n} = \cA_{m,n} P_n,\quad m, n\in\N,$$
and there are constants $\lambda, D > 0$ such that, for every $x \in X$ and $n,m \in \N$, we have
\begin{equation}\label{eq:dich-1}
\|\cA_{m,n} P_n x\|_m \le D \left(\mu_m/\mu_n\right)^{-\lambda} \|x\|_n, \quad \text{for} \ m \ge n
\end{equation}
and
\begin{equation}\label{eq:dich-2}
\|\cA_{m,n}Q_n x\|_m \le D \left(\mu_n/\mu_m\right)^{-\lambda} \|x\|_n, \quad \text{for} \ m \le n,
\end{equation}
where $Q_m=\Id-P_m$ is the complementary projection and, for $m \le n$, we use the notation $\cA_{m,n}=(\cA_{n,m})^{-1}: \ker P_n \to \ker P_m$.

\section{Admissibility with respect to sequences of norms}

Let $\|\cdot\|_\infty$ denote as usual the supremum norm: $\|x\|_\infty:=\sup_{m \in \N} \|x_m\|_m$. We define the space
$$Y=\{x=(x_m)\subset X: \|x\|_\infty < \infty\},$$
and, given a subspace $Z \subset X$, we define
$$Y_Z=\{x=(x_m)\subset Y: x_1 \in Z\}.$$
It is easy to see that $Y_Z$ is closed subespace of $Y$. We use the notation $Y_0=Y_{\{0\}}$.

For a given discrete growth rate $(\mu_n)_{n \in \N}$, define a sequence $(\phi^\mu_n)_{n \in \N}$ by $\phi^\mu_n=\mu_n/\mu'_n$ where $\mu'_n=\mu_{n+1}-\mu_n$ and a linear operator,
$$T^\mu_Z:(\cD(T^\mu_Z),\|\cdot\|_{T^\mu_Z}) \to Y_0,$$
by
\begin{equation}\label{eq:def-TmuZ}
(T^\mu_Z x)_1=0 \quad \text{ \and } \quad (T^\mu_Z x)_m=\phi^\mu_m(x_m-A_{m-1}x_{m-1}), \quad \ m \ge 1,
\end{equation}
where
$$\cD(T^\mu_Z):=\{x=(x_m)\subset Y_Z: T^\mu_Z x \in Y_0\},$$
and the norm $\|\cdot\|_{T^\mu_Z}$ is given by
$$\|x\|_{T^\mu_Z}:=\|x\|_\infty+\|T^\mu_Zx\|_\infty.$$

Before stating our main results, we will begin with a technical lemma that will be very useful in some estimates. The proof of the result relies on a result concerning the smooth global approximation of a Lipschitz functin by differentiable functions obtained in~\cite{Czarnecki.Rifford.TAMS.2006}.
\begin{lemma}\label{prop:ESTIMATES}
Let $(\mu_n)_{n\in\N}$ be a discrete growth rate.
If $\alpha \in \ ]0,1[ \ \cup \ ]1,+\infty[$, for each $r,s \in \N$ with $r \ge s >1$ we have
\begin{equation}\label{prop:ESTIMATES-est-le-<0ne1}
\frac{\mu_{r+1}^{1-\alpha}-\mu_s^{1-\alpha}}{1-\alpha} \le \sum_{k=s}^r \mu_k^{-\alpha}\mu_k' \le \frac{\mu_r^{1-\alpha}-\mu_{s-1}^{1-\alpha}}{1-\alpha}
\end{equation}
and for $\alpha=1$ we have
\begin{equation}\label{prop:ESTIMATES-est-=1}
\ln \frac{\mu_{r+1}^+}{\mu_{s}^+} \le \sum_{k=s}^{r} \, \mu_k'/\mu_k \le \ln \frac{\mu_r^-}{\mu_{s-1}^-}.
\end{equation}
\end{lemma}

\begin{proof}
By Theorem~B in \cite{Czarnecki.Rifford.TAMS.2006}, given a locally Lipschitz function $f:\R \to \R$ there is a sequence of differentiable functions
$(f_n)_{n \in \N}$ such that
\begin{equation}\label{eq:conditions-teo-aprox}
\lim_{n \to +\infty} \|f_n-f\|_\infty \to 0 \quad \text{ and } \quad \lim_{n \to +\infty} d_H\left(\mathcal G(f'_n),\mathcal G(\partial f)\right) \to 0,
\end{equation}
where $d_H$ denotes the Hausdorff distance, $\mathcal G(g)$ denotes the graph of $g$ and finally $\partial f$ is Clarke’s generalized derivative of $f$.

Given a discrete growth rate $(\mu_n)_{n\in\N}$, we will apply the result to the locally Lipschitz function $f:\R \to \R$ given by
\[
f(t)=
\begin{cases}
\mu_1 & \text{if} \ t \le 1\\
\mu_n + (\mu_{n+1}-\mu_n)(t-n) & \text{if} \ n < t \le n+1, \ n \in \N
\end{cases}.
\]
Note that Clarke’s generalized derivative of $f$ is given by
\[
\partial f(t)=
\begin{cases}
0 & \text{if} \ t < 1\\
\mu_{n+1}-\mu_n & \text{if} \ n < t < n+1, \ n \in \N\\
I_n & \text{if} \ n \in \N
\end{cases},
\]
where $I_n$ is the closed interval with endpoints $p=\mu_{n+2}-\mu_{n+1}$ and $q=\mu_{n+1}-\mu_n$.

By the result mentioned above, given $\eps \in (0,1)$, there is a differentiable function $f_\eps$ verifying
\begin{equation}\label{eq:conditions-teo-aprox2}
\|f-f_\eps\|_\infty < \eps/2 \quad \text{ and } \quad d_H\left(\mathcal G(\partial f),\mathcal G(f'_\eps)\right)<\eps/2.
\end{equation}

Define functions $\mu_\eps^-,\mu_\eps^+: \R \to \R$ by $\mu_\eps^\pm(t)=f_\eps(t)\pm \eps/2$. It follows that, for $t\in \R\setminus \N$,
we have $|\mu_\eps^\pm(t)-f(t)|<\eps$, $\mu_\eps^-(t)<f(t)<\mu_\eps^+(t)$ and $d_H\left(\mathcal G(\partial f),\mathcal G(f'_\eps)\right)<\eps/2$.

Define
$$\mathcal C= \bigcup_{i=1}^{+\infty} [n-\eps,n+\eps].$$
Let $f':\R\setminus\N \to \R$ be the classical derivative of $f$. We have
$$|(\mu_\eps^\pm)'(t)-f'(t)|=|(\mu_\eps^\pm)'(t)-\partial f(t)|<\eps,$$
for $t \in \R\setminus \mathcal C$ and
$$|(\mu_\eps^\pm)'(t)- f'(t)| < |f'_\eps(t)- f'(t)| +\frac{\eps}{2}
< \mu_{n+1}-\mu_n+2\eps
< (\mu_\eps^\pm)'(t)+\mu_{n+1}-\mu_n+2\eps,$$
for $t \in [n-\eps,n[ \, \cup \, ]n,n+\eps]$. Thus, if $s,r \in \N$ with $1<s<r$, we have
\[
\begin{split}
\sum_{k=s}^{r} \mu_k^{-\alpha}\mu_k'
& \le \int_{s-1}^{r} \left(\mu_\eps^-(t)\right)^{-\alpha} f'(t)\dt\\
& \le \int_{s-1}^{r} \left(\mu_\eps^-(t)\right)^{-\alpha} (\mu_\eps^-)'(t)\dt + \eps(r-s+2)\\
& \phantom{\le} + 2\eps\left(\sup_{n \in \{s,\ldots,r\}}\{\mu_{n+1}-\mu_n\}+2\eps\right)(r-s+2) \\
& = \frac{\mu_\eps^-(r)^{1-\alpha}-\mu_\eps^-(s-1)^{1-\alpha}}{1-\alpha}\\
& \phantom{\le} +\eps(r-s+2)[1+ 2\sup_{n \in \{s,\ldots,r\}}\{\mu_{n+1}-\mu_n\}+4\eps] \\
& \le \frac{(f_\eps(r)-\eps/2)^{1-\alpha}-(f_\eps(s-1)-\eps/2)^{1-\alpha}}{1-\alpha}+\eps c_{r,s}.
\end{split}
\]
where
$$c_{r,s}=(r-s+2)\left[5+ 2\sup_{n \in \{s,\ldots,r\}}\{\mu_{n+1}-\mu_n\}\right].$$
Note that $f'$ is only defined on $\R\setminus \N$ but this is enough for the integral since $[s,r] \cap \N$ is a zero measure set.
Thus, for $\alpha \in \, ]0,1[ $, we have
\[
\begin{split}
\sum_{k=s}^{r} \mu_k^{-\alpha}\mu_k'
& \le \frac{f(r)^{1-\alpha}-(f(s-1)-\eps)^{1-\alpha}}{1-\alpha}+\eps c_{r,s}\\
& = \frac{\mu_{r}^{1-\alpha}-(\mu_{s-1}-\eps)^{1-\alpha}}{1-\alpha}+\eps c_{r,s}
\end{split}
\]
and, for $\alpha >1$, we have
\[
\begin{split}
\sum_{k=s}^{r} \mu_k^{-\alpha}\mu_k'
& \le \frac{(f(r)-\eps)^{1-\alpha}-f(s-1)^{1-\alpha}}{1-\alpha}+\eps c_{r,s}\\
& = \frac{(\mu_{r}-\eps)^{1-\alpha}-\mu_{s-1}^{1-\alpha}}{1-\alpha}+\eps c_{r,s}.
\end{split}
\]
From the arbitrariness of $\eps>0$, we get the result for $\alpha \in ]0,1[ \, \cup \, ]1,+\infty[$ when $s>1$.

Taking now $\alpha=1$, and using the integral to estimate the sum as before, we get, if $s>1$,
\[
\begin{split}
\sum_{k=s}^{r} \, \mu_k'/\mu_k
& \le \left[\ln \mu_\eps^-(t) \right]_{s-1}^{r}+\eps c_{r,s}
= \ln \frac{\mu_r^-}{\mu_{s-1}^-}+\eps c_{r,s}.
\end{split}
\]
Again, from the arbitrariness of $\eps>0$, we get the result.

Similarly, we have, for $\alpha \in ]0,1[ \, \cup \, ]1,+\infty[$,
\[
\begin{split}
\sum_{k=s}^{r} \mu_k^{-\alpha}\mu_k'
& \ge \int_{s}^{r+1} \left(\mu_\eps^+(t)\right)^{-\alpha} f'(t)\dt\\
& \ge \int_{s}^{r+1} \left(\mu_\eps^+(t)\right)^{-\alpha} (\mu_\eps^+)'(t)\dt - \eps(r-s+2)\\
& \phantom{\le} - 2\eps\left(\sup_{n \in \{s,\ldots,r\}}\{\mu_{n+1}-\mu_n\}+2\eps\right)(r-s+2) \\
& = \frac{\mu_\eps^+(r+1)^{1-\alpha}-\mu_\eps^+(s)^{1-\alpha}}{1-\alpha}\\
& \phantom{\le} -\eps(r-s+2)[1+ 2\sup_{n \in \{s,\ldots,r\}}\{\mu_{n+1}-\mu_n\}+4\eps] \\
& = \frac{(f_\eps(r+1)-\eps/2)^{1-\alpha}-(f_\eps(s)-\eps/2)^{1-\alpha}}{1-\alpha}-\eps c_{r+1,s+1}\\
& = \frac{(f(r+1)-\eps)^{1-\alpha}-f(s)^{1-\alpha}}{1-\alpha}-\eps c_{r+1,s+1}.
\end{split}
\]
Finally, if $\alpha=1$ and $s>1$, estimating the sum as before, we obtain
\[
\begin{split}
\sum_{k=s}^{r} \, \mu_k'/\mu_k
& \ge \left[\ln \mu_\eps^+(t) \right]_s^{r+1}+\eps c_{r+1,s+1}
= \ln \frac{\mu_{r+1}^+}{\mu_{s}^+}+\eps c_{r+1,s+1}.
\end{split}
\]
Once more, from the arbitrariness of $\eps>0$, we get the result.
\end{proof}

We have the following remark.

\begin{remark}
In Proposition~\ref{prop:ESTIMATES} we have obtained estimates only in the case $s>1$.
If $s=1$, estimates~\eqref{prop:ESTIMATES-est-le-<0ne1} and~\eqref{prop:ESTIMATES-est-=1}
allow us to obtain, for $\alpha \in ]0,1[ \, \cup \, ]1,+\infty[$,
\begin{equation}\label{remark:ESTIMATES-est-le-<0ne1-1}
\phi^\mu_1\mu_1^{1-\alpha}+\frac{\mu_{r+1}^{1-\alpha}-\mu_2^{1-\alpha}}{1-\alpha} \le\sum_{k=1}^r \mu_k^{-\alpha}\mu_k' \le \phi^\mu_1\mu_1^{1-\alpha}+\frac{\mu_r^{1-\alpha}-\mu_1^{1-\alpha}}{1-\alpha}
\end{equation}
and, when $\alpha=1$, we get
and, as before,
\begin{equation}
\phi^\mu_1+\ln \frac{\mu_{r+1}^+}{\mu_2^+} \le \sum_{k=1}^{r} \, \mu_k'/\mu_k \le \phi^\mu_1+\ln \frac{\mu_r^-}{\mu_1^-}.
\end{equation}
\end{remark}

The arguments used in Proposition 1 in~\cite{D.MN.2019} show that, also in the present context, $T^\mu_Z$ is a closed linear operator. It is easy to check that $(\cD(T^\mu_Z),\|\cdot\|_{T^\mu_Z})$ is a Banach space. Notice that $\cD(T^\mu_Z)=\{x=(x_m)\subset Y_Z: (T^\mu_Z x)_1=0\}$.

\begin{theorem}\label{teo:invertibility}
Let $\mu$ be a differentiable growth rate.
If the sequence $(A_m)_{m \in \N}$ admits a $\mu$-dichotomy with respect to a sequence of norms $(\|\cdot\|_m)_{m \in \N}$, then the operator $T^\mu_{\im Q_1}$ is invertible.
\end{theorem}

\begin{proof}

We may assume without loss of generality that $\lambda \in (0,1)$. To prove that $T_{\im Q}$ in surjective, let
$y=(y_n)_{n \in \N} \in Y_0$ and define a sequence $x=(x_n)_{n \in \N}$ in $X$ by
\[
x_n=\sum_{k=1}^n \frac{1}{\phi^\mu_k} \cA_{n,k}P_k y_k - \sum_{k=n+1}^\infty \frac{1}{\phi^\mu_k} \cA_{n,k}Q_k y_k.
\]
Using~\eqref{eq:dich-1} and~\eqref{remark:ESTIMATES-est-le-<0ne1-1}, we have, for each $n \in \N$,
\begin{equation}\label{eq:maj-seq-x-P}
\begin{split}
\left\| \sum_{k=1}^n \frac{1}{\phi^\mu_k} \cA_{n,k}P_k y_k \right\|_n
& \le \sum_{k=1}^n \frac{1}{\phi^\mu_k} \left\|\cA_{n,k}P_k y_k \right\|_n\\
&= D \sum_{k=1}^n \frac{\mu'_k}{\mu_k} \left(\frac{\mu_n}{\mu_k}\right)^{-\lambda} \left\|y_k \right\|_k\\
& \le D \mu_n^{-\lambda} \left\|y \right\|_\infty \sum_{k=1}^n \mu_k^{-1+\lambda}\mu'_k \\
& \le D  \mu_n^{-\lambda} \left\|y \right\|_\infty \left(\mu_1^\lambda\phi_1^\mu+\mu_n^\lambda/\lambda-\mu_1^\lambda/\lambda\right)\\
&\le D(\phi_1^\mu+1/\lambda) \left\|y \right\|_\infty
\end{split}
\end{equation}
and, using~\eqref{eq:dich-2} and~\eqref{prop:ESTIMATES-est-le-<0ne1}, we have, for each $n \in \N$,
\begin{equation}\label{eq:maj-seq-x-Q}
\begin{split}
\left\| \sum_{k=n+1}^\infty \frac{1}{\phi^\mu_k} \cA_{n,k}Q_k y_k \right\|_n
& \le \sum_{k=n+1}^\infty \frac{1}{\phi^\mu_k} \left\|\cA_{n,k}Q_k y_k \right\|_n\\
& \le D \sum_{k=n+1}^\infty \frac{\mu_k'}{\mu_k} \left(\frac{\mu_k}{\mu_n}\right)^{-\lambda} \left\|y_k \right\|_k\\
& \le D \mu_n^\lambda \left\|y \right\|_\infty \sum_{k=n+1}^\infty \mu_k^{-1-\lambda}\mu_k'\\
& \le D \mu_n^{\lambda} \left\|y \right\|_\infty \mu_n^{-\lambda}/\lambda
\le D/\lambda \left\|y \right\|_\infty
\end{split}
\end{equation}
By~\eqref{eq:maj-seq-x-P} and ~\eqref{eq:maj-seq-x-Q}, we conclude that $\sup_{n \in \N} \|x_n\|_n \le D(\phi_1^\mu+2/\lambda)\left\|y \right\|_\infty$ and, since $y_1=0$, we conclude that $x_1 \in \im Q_1$. Thus $x \in \cD(T_{\im Q})$. Simmilar computations to those in the proof of Theorem 2 in~\cite{D.MN.2019} show that $T_{\im Q} x=y$. Therefore $T_{\im Q}$ is surjective.

To prove that $T_{\im Q}$ is injective, let $u=(u_n)_{n \in \N},v=(v_n)_{n \in \N} \in \cD(T_{\im Q_1})$ be such that $T_{\im Q_1}u=T_{\im Q_1}v=0$. We have that
$$\|u_1-v_1\|_1=\|\cA_{1,n}(u_n-v_n)\|_1\le D \mu_n^{-\lambda}\|u-v\|_\infty \to 0,$$
as $n \to \infty$, and we conclude that $u_1=v_1$. Thus $u_n=\cA(n,1)u_1=\cA(n,1)v_1=v_n$, for each $n \in \N$. We conclude that $u=v$ and $T_{\im Q}$ is injective.
\end{proof}

\begin{theorem}\label{teo:exist-dich}
Let $\mu$ be a discrete growth rate. Assume that there is an increasing sequence $(q_n)$ with $q_n\ge n+1$, $n \in \N$, and constants $L_1,L_2>1$ such that, for all $n \in\ \N$,
\begin{equation}\label{eq:COND-minimal-growth}
L_1 \le \frac{\mu_{q_n}}{\mu_n} \le L_2.
\end{equation}
If $T^\mu_Z$ is an invertible operator for some closed subespace $Z \subset X$ and there are constants $\lambda,M>0$ such that
\begin{equation}\label{eq:teo:exist-dich:exponential-bound}
\|\cA_{m,n}x\|_m\le M(\mu_m/\mu_n)^\lambda\|x\|_n,
\end{equation}
for all $m \ge n$ and $x \in X$, then $(A_m)_{m \in \N}$ admits a $\mu$-dichotomy with respect to the sequence of norms $(\|\cdot\|_m)_{m \in \N}$.
\end{theorem}

\begin{proof}
For each $n \in \N$, consider the following subespaces of $X$:
$$X_n=\left\{x \in X: \sup_{m \ge n} \|\cA_{m,n} x\|<\infty\right\} \quad \text{ and } \quad Z_n=\cA_{n,1}Z.$$

It is easy to deduce that, for each $n \in \N$, we have:
$$A_nX_n \subset X_{n+1} \quad \text{ and } \quad A_n Z_n = Z_{n+1}.$$

\begin{lemma}
We have $X=X_n\oplus Z_n$, for each $n \in \N$.
\end{lemma}

\begin{proof}[Proof of the lemma]
For $v \in X$, we define the sequences $u^1=(v,0,0,\ldots)$ and $w^1=(0,-\phi_2^\mu A_1v,0,\ldots)$. Note that
$(T^\mu_Z u^1)_1=0=w^1_1$, $(T^\mu_Zu^1)_2=\phi_2^\mu(u^1_2-A_1u^1_1)=-\phi_2^\mu A_1v=w^1_2$ and $(T^\mu_Zu^1)_m=0=w^1_m$ for $m \ge 3$. Thus $T^\mu_Zu^1=w^1$.
Since $T^\mu_Z$ is surjective and $w^1 \in Y_0$, we conclude that there is $z^1 \in Y_Z$ such that $T^\mu_Z z^1=w^1$ and, in particular, $z^1_1\in Z_1$.
We have,
\[
\begin{split}
u^1_m-z^1_m
& =\frac{1}{\phi^\mu_m}w^1_m-A_{m-1}u^1_{m-1}-\frac{1}{\phi^\mu_m}w^1_m+A_{m-1}z^1_{m-1}\\
& =A_{m-1} (u^1_{m-1}-z^1_{m-1})\\
& = \cdots \\
& = \cA_{m,1}(v-z^1_1),
\end{split}
\]
for $m \in \N$. Since $u^1-z^1 \in Y$, we have $\sup_{m \ge 1}\|\cA_{m,1}(v-z^1_1)\|<+\infty$ and thus $v-z^1_1\in X_1$.
Therefore $v=v-z^1_1+z^1_1 \in X_1+Z_1$.

Letting now $k>1$, we define the sequences $u^k=(u^k_i)_{i \in \N}$ and $w^k=(w^k_i)_{i \in \N}$ by
\[
u^k_i
=
\begin{cases}
v & \quad \text{if} \, i=k\\
0 & \quad \text{if} \, i\ne k
\end{cases}
\quad\quad
\text{and}
\quad\quad
w^k_i
=
\begin{cases}
\phi^\mu_k v & \quad \text{if} \, i=k\\
-\phi^\mu_{k+1} A_k v & \quad \text{if} \, i=k+1\\
0 & \quad \text{if} \, i\ne k, k+1
\end{cases}.
\]
It is easy to check that $T^\mu_Zu^k=w^k$. Since $T^\mu_Z$ is surjective and $w^k \in Y_0$, we conclude that there is $z^k \in Y_Z$ such that $T^\mu_Z z^k=w^k$ and, in particular, $z^k_k\in Z_k$.
We have,
\[
\begin{split}
u^k_m-z^k_m
& =\frac{1}{\phi^\mu_m}w^k_m+A_{m-1}u^k_{m-1}-\frac{1}{\phi^\mu_m}w^k_m-A_{m-1}z^k_{m-1}\\
& =A_{m-1} (u^k_{m-1}-z^k_{m-1})\\
& = \cdots \\
& = \cA_{m,k}(v-z^k_k),
\end{split}
\]
for $m \in \N$. Since $u^k-z^k \in Y$, we have $\sup_{m \ge 1}\|\cA_{m,k}(v-z^k_k)\|<+\infty$ and thus $v-z^k_k\in X_k$.
Therefore $v=v-z^k_k+z^k_k \in X_k+Z_k$. We have concluded that $v \in X_n+Z_n$ for each $n \in \N$.

Let now $m \in \N$ and $v \in X_m \cap Z_m$. Since $v \in Z_m=\cA_{m,1}Z_1$, there is $z \in Z_1=Z$ such that $v=\cA_{m,1}z$.
Define a sequence $x=(x_n)_{n \in \N}$ by $x_n=\cA_{n,1}z$, $n \in \N$. Since $v \in X_m$ and $z \in Z$, it is easy to conclude that $x \in Y_Z$.
Since $(T^\mu_Zx)_1=0$ and $(T^\mu_Zx)_k=\phi^\mu_k(x_k-A_{k-1}x_{k-1})=\phi^\mu_k(\cA_{k,1}z-A_{k-1}\cA_{k-1,1}z)=0$, we conclude that $T^\mu_Zx=0$. By the invertibility of
$T^\mu_Z$, we have $x=0$ and consequently $v=0$. Thus $X_m \cap Z_m = \{0\}$.

We conclude that $x=X_m \oplus Z_m$, for any $m \in \N$.
\end{proof}

The same arguments used in the proof of Lemma 3 in~\cite{D.MN.2019} allow us to conclude that, for each $n \in \N$, the linear map $A_n|_{Z_n}: Z_n \to Z_{n+1}$ is invertible.


\begin{lemma}\label{lemma:estimate-cA-X}
There are $D_1,a>0$ such that $\|\cA_{m,n}x\|_m \le D_1(\mu_m/\mu_n)^{-a}\|x\|_n$,
for $m\ge n$ and $x\in X_n$.
\end{lemma}

\begin{proof}
Consider the sequences $y=(y_n)_{n \in \N}$ and $x=(x_n)_{n \in \N}$ given by
\[
y_k=
\begin{cases}
\dfrac{\cA(k,n)x}{\|\cA(k,n)x\|_k}, & \quad \text{if} \ n+1\le k\le m\\
0, & \quad \text{otherwise}
\end{cases}
\]
and
\[
x_k=
\begin{cases}
0, & \quad \text{if} \ 1\le k\le n\\[2mm]
\displaystyle \sum_{j=n+1}^k \frac{\cA(k,n)x}{\phi^\mu_j\|\cA(j,n)x\|_j}, & \quad \text{if} \ n+1\le k\le m\\[2mm]
\displaystyle \sum_{j=n+1}^m \frac{\cA(k,n)x}{\phi^\mu_j\|\cA(j,n)x\|_j}, & \quad \text{if} \ k> m
\end{cases}.
\]
It is easy to check that $y \in Y_0$, $x \in Y_{\im Q_1}$ and $T^\mu_{\im Q_1}x=y$. Using this, we obtain
\[
\begin{split}
\|x\|_\infty \le \|x\|_{T^\mu_{\im Q_1}}
&=\|(T^\mu_{\im Q_1})^{-1}y\|_{T^\mu_{\im Q_1}}\le\|(T^\mu_{\im Q_1})^{-1}\|\|y\|_\infty\\
&\le\|(T^\mu_{\im Q_1})^{-1}\|\le\|(T^\mu_{\im Q_1})^{-1}\|,
\end{split}
\]
and thus
\begin{equation}\label{eq:TmuQ1}
\|(T^\mu_{\im Q_1})^{-1}\| \ge \|x\|_\infty \ge \|x_m\|_m=\|\cA_{m,n}x\|_m \sum_{j=n+1}^m \frac{1}{\phi^\mu_j\|\cA_{j,n}x\|_j}.
\end{equation}
Using~\eqref{eq:teo:exist-dich:exponential-bound} and~\eqref{eq:COND-minimal-growth} we have
$$\|\cA_{j,n}x\|_j\le M(\mu_j/\mu_n)^\lambda\|x\|_n.$$
Thus, by~\eqref{eq:TmuQ1},
\begin{equation}\label{eq:bound-by-norm-of-inverse}
\|(T^\mu_{\im Q_1})^{-1}\| \ge \|\cA_{m,n}x\|_m \sum_{j=n+1}^m \frac{1}{\phi^\mu_jM(\mu_j/\mu_n)^\lambda\|x\|_n}
\ge \frac{\|\cA_{m,n}x\|_m}{M\|x\|_n}\sum_{j=n+1}^m \frac{\mu_n^\lambda}{\phi^\mu_j\mu_j^\lambda}.
\end{equation}

Therefore, for $m \ge q_{n+1}-1$,
\[
\begin{split}
\|(T^\mu_{\im Q_1})^{-1}\|
& \ge \frac{\|\cA_{m,n}x\|_m}{M\|x\|_n}\sum_{j=n+1}^{q_{n+1}-1} \frac{\mu_n^\lambda}{\phi^\mu_j\mu_j^\lambda}\\
& \ge \frac{\|\cA_{m,n}x\|_m}{M\|x\|_n}\, \mu_n^\lambda\sum_{j=n+1}^{q_{n+1}-1} \frac{1}{\phi^\mu_j\mu_j^\lambda}\\
& \ge \frac{1}{\lambda M} \left(\frac{\mu_n}{\mu_{n+1}}\right)^\lambda\left(1-\left(\frac{\mu_{n+1}}{\mu_{q_{n+1}}}\right)^\lambda\right) \frac{\|\cA_{m,n}x\|_m}{\|x\|_n}\\
& \ge \frac{1}{\lambda M} \left(\frac{\mu_n}{\mu_{q_n}}\right)^\lambda\left(1-\left(\frac{\mu_{n+1}}{\mu_{q_{n+1}}}\right)^\lambda\right) \frac{\|\cA_{m,n}x\|_m}{\|x\|_n}\\
& \ge \frac{1-L_1^{-\lambda}}{\lambda ML_2^\lambda} \frac{\|\cA_{m,n}x\|_m}{\|x\|_n}.
\end{split}
\]
We conclude that, for $m \ge q_{n+1}-1$, we have
\[
 \|\cA_{m,n}x\|_m
\le \lambda M L_2^\lambda L_1^\lambda(L_1^\lambda-1)^{-1} \|(T^\mu_{\im Q_1})^{-1}\| \, \|x\|_n.
\]

Using~\eqref{eq:teo:exist-dich:exponential-bound}, we have, for $n \le m < q_{n+1}-1$,
\[
\begin{split}
\|\cA_{m,n}x\|_m
& \le M(\mu_m/\mu_n)^\lambda\|x\|_n\le M(\mu_{q_{n+1}-1}/\mu_n)^\lambda\|x\|_n \\
& \le M(\mu_{q_{n+1}-1}/\mu_{n+1})^\lambda (\mu_{n+1}/\mu_{n})^\lambda\|x\|_n\\
& \le M(\mu_{q_{n+1}}/\mu_{n+1})^\lambda (\mu_{q_n}/\mu_{n})^\lambda\|x\|_n
\le ML_2^{2\lambda}\|x\|_n.
\end{split}
\]

We conclude that, for $m \ge n$,
\begin{equation}\label{eq:bound-cA-M2}
\|\cA_{m,n}x\|_m \le M_2\|x\|_n,
\end{equation}
where
\[
M_2=M L_2^\lambda \max\left\{\lambda L_1^\lambda(L_1^\lambda-1)^{-1} \|(T^\mu_{\im Q_1})^{-1}\|, \, L_2^\lambda \right\}.
\]

Taking into account~\eqref{prop:ESTIMATES-est-=1}, for $m > n$ we have
\begin{equation}\label{eq:bound-importamt-sum}
\begin{split}
\sum_{j=n+1}^m \frac{1}{\phi^\mu_j} \ge \ln \frac{\mu_{m+1}}{\mu_{n+1}}.
\end{split}
\end{equation}

Since $\mu$ is a growth rate, by~\eqref{eq:TmuQ1},~\eqref{eq:bound-cA-M2} and~\eqref{eq:bound-importamt-sum}, we have
\[
\begin{split}
\|(T^\mu_{\im Q_1})^{-1}\|
& \ge \|\cA_{m,n}x\|_{m} \sum_{j=n+1}^{m} \frac{1}{\phi^\mu_j\|\cA_{j,n}x\|_j}
\ge \frac{\|\cA_{m,n}x\|_{m}}{M_2 \|x\|_n} \ln \frac{\mu_{m+1}}{\mu_{n+1}}.\\
& \ge \frac{\|\cA_{m,n}x\|_{m}}{M_2 \|x\|_n} \ln \frac{\mu_{m}}{L_2\mu_{n}}.\\
\end{split}
\]
Taking into account that
$$ M_2 \|(T^\mu_{\im Q_1})^{-1}\| \left(\ln \frac{\mu_m}{L_2 \mu_n}\right)^{-1} \le \e^{-1} \quad
\Leftrightarrow \quad \mu_m  \ge N_0 \mu_n,$$
where
$$N_0=L_2\exp\{\e M_2\|(T^\mu_{\im Q_1})^{-1}\|\},$$
we conclude that
\begin{equation}\label{eq:crucial}
\|\cA_{m,n}x\|_m \le M_2 \|(T^\mu_{\im Q_1})^{-1}\| \left(\ln \frac{\mu_{m+1}}{\mu_{n+1}}\right)^{-1} \|x\|_n
\le \e^{-1} \|x\|_n,
\end{equation}
for $m,n \in \N$ such that $\mu_m \ge N_0\mu_n$.

Define $\gamma(n,0)=n$ and $\gamma(n,k)=q_{\gamma(n,k-1)}$ for $k\in \N$.
According to~\eqref{eq:COND-minimal-growth}, we have, for all $k,n \in \N$
\[
\begin{split}
\frac{\mu_{\gamma(n,k)}}{\mu_n} = \frac{\mu_{\gamma(n,k)}}{\mu_{\gamma(n,k-1)}}
\frac{\mu_{\gamma(n,k-1)}}{\mu_{\gamma(n,k-2)}} \cdots \frac{\mu_{\gamma(n,2)}}{\mu_{\gamma(n,1)}}
 \ge \frac{\mu_{q_{\gamma(n,k-1)}}}{\mu_{\gamma(n,k-1)}}
  \frac{\mu_{q_{\gamma(n,k-2)}}}{\mu_{\gamma(n,k-2)}}\cdots \frac{\mu_{q_{\gamma(n,1)}}}{\mu_{\gamma(n,1)}}
  \ge L_1^k.
\end{split}
\]
Thus, there is $K_0 \in \N$ sufficiently large such that, for all $n \in \N$,
\begin{equation}\label{eq:conversion-m.n-mum.mun}
\frac{\mu_{\gamma(n,K_0)}}{\mu_n} \ge L_1^{K_0} \ge N_0.
\end{equation}
We conclude that~\eqref{eq:crucial} holds for all $m,n \in \N$, $m \ge n$, such that $\mu_m \ge\mu_{\gamma(n,K_0)}$.

Assume now that $m,n \in \N$, $m \ge n$ with no other restrictions and consider the largest $\ell \in \N \cup \{0\}$ such that $\mu_m \ge \mu_{\gamma(n,\ell K_0)}$. Since, for $p=0,\ldots,\ell-1$, we have
\[
\mu_{\gamma(n,(\ell-p) K_0)}\ge L_1^{K_0} \mu_{\gamma(n,(\ell-p-1) K_0)}
\ge N_0 \mu_{\gamma(n,(\ell-p-1) K_0)}
\]
we conclude, using~\eqref{eq:crucial}, that
\begin{equation}\label{eq:aux-bound-cAmn}
\begin{split}
& \|\cA_{m,n} x\|_m\\
& = \|\cA_{m,\gamma(n,\ell K_0)}\cA_{\gamma(n,\ell K_0),n} x\|_m\\
& = \|\cA_{m,\gamma(n,\ell K_0)}\cA_{\gamma(n,\ell K_0),\gamma(n,(\ell-1) K_0)}
\cdots \cA_{\gamma(n,K_0),n}x\|_m\\
& \le M\left(\mu_m/\mu_{\gamma(n,\ell K_0)}\right)^\lambda \|\cA_{\gamma(n,\ell K_0),\gamma(n,(\ell-1) K_0)}
\cdots \cA_{\gamma(n,K_0),n}x\|_{\gamma(n,\ell K_0)}\\
& \le ML_2^\lambda \e^{-\ell} \|x\|_n\\
\end{split}
\end{equation}
Noting that
\[
\dfrac{\mu_m}{\mu_n} \le \dfrac{\mu_{\gamma(n,(\ell+1) K_0)}}{\mu_n} = \dfrac{\mu_{\gamma(n,(\ell+1) K_0)}}{\mu_{\gamma(n,(\ell+1) K_0-1)}}
\dfrac{\mu_{\gamma(n,(\ell+1) K_0-1)}}{\mu_{\gamma(n,(\ell+1) K_0-2)}}\cdots
\dfrac{\mu_{\gamma(n,1)}}{\mu_n} \le L_2^{(\ell+1)K_0}
\]
we get
\begin{equation}\label{eq:aux-bound-cAmn-2}
\ell > \dfrac{\log \left(\mu_m/\mu_n\right)}{K_0\log L_2}-1
\quad \Leftrightarrow \quad
\e^{-\ell} < \e \, \left(\mu_m/\mu_n\right)^{-1/(K_0 \log L_2)}.
\end{equation}
It follows from~\eqref{eq:aux-bound-cAmn} and~\eqref{eq:aux-bound-cAmn-2} that, for $m \ge n$,
\[
\begin{split}
\|\cA_{m,n} x\|_m
& \le M L_2^\lambda \e \, \left(\mu_m/\mu_n\right)^{-1/(K_0 \log L_2)} \|x\|_n
\end{split}
\]
and we obtain the estimate $\|\cA_{m,n}x\|_m \le D_1 \left(\mu_m/\mu_n\right)^{-a} \, \|x\|_n$ with $D_1=\e M L_2^\lambda$ and $a=1/(K_0 \log L_2)$. The lemma is proved.
\end{proof}

\begin{lemma}\label{lemma:estimate-cA-Z}
There are $D_2,b>0$ such that $\|\cA_{m,n}x\|_m \ge D_2(\mu_n/\mu_m)^{-b}\|x\|_n$,
for $m\le n$ and $x\in Z_n$.
\end{lemma}

\begin{proof}
Let $z \in Z \setminus\{0\}$ and consider the sequences $y=(y_n)_{n \in \N}$ and $x=(x_n)_{n \in \N}$ given by
\[
y_k=
\begin{cases}
0, & \quad \text{if} \ k=1\\
-\dfrac{\cA_{k,1}z}{\|\cA_{k,1}z\|_k}, & \quad \text{if} \ 2 \le k\le m\\
0, & \quad \text{otherwise}
\end{cases}
\]
and
\[
x_k=
\begin{cases}
\displaystyle \sum_{j=k+1}^n \frac{\cA_{k,1}z}{\phi^\mu_j\|\cA_{j,1}z\|_j}, & \quad \text{if} \ 1\le k\le n-1\\[2mm]
0, & \quad \text{if} \ k \ge n
\end{cases}.
\]
We have $y \in Y_0$ and $x \in Y_{\im Q_1}$ and it is easy to check that $T^\mu_{\im Q_1}x=y$ and, like in~\eqref{eq:bound-by-norm-of-inverse}, we
get $\|x\|_\infty\le\|T_{\im Q_1}^{-1}\|$.
For each $1\le k \le n-1$, we have
\[
\|(T^\mu_{\im Q_1})^{-1}\| \ge \|\cA_{k,1} z \|_k \sum_{j=k+1}^n \frac{1}{\phi^\mu_j\|\cA_{j,1}z\|_j}
\]
and, letting $n \to +\infty$, we obtain, for $k \in \N$ and $z \in Z\setminus\{0\}$,
\[
\|(T^\mu_{\im Q_1})^{-1}\| \ge \|\cA_{k,1} z \|_k \sum_{j=k+1}^\infty \frac{1}{\phi^\mu_j\|\cA_{j,1}z\|_j}.
\]

We have, by~\eqref{prop:ESTIMATES-est-le-<0ne1}, ~\eqref{eq:COND-minimal-growth} and~\eqref{eq:teo:exist-dich:exponential-bound},
\[
\begin{split}
\frac{1}{\|\cA_{n,1} z \|_n}
& \ge \frac{1}{\|(T^\mu_{\im Q_1})^{-1}\|} \sum_{j=n+1}^\infty \frac{1}{\phi^\mu_j\|\cA_{j,1}z\|_j}\\
& \ge \frac{1}{\|(T^\mu_{\im Q_1})^{-1}\|} \sum_{j=m+1}^{q_m} \frac{1}{\phi^\mu_j\|\cA_{j,m}\cA_{m,1}z\|_j}\\
& \ge \frac{1}{\|(T^\mu_{\im Q_1})^{-1}\|} \sum_{j=m+1}^{q_m} \frac{1}{\phi^\mu_j M (\mu_j/\mu_m)^\lambda \|\cA_{m,1}z\|_m}\\
& \ge \frac{1}{M\|(T^\mu_{\im Q_1})^{-1}\| \, \|\cA_{m,1}z\|_m} \sum_{j=m+1}^{q_m} \frac{1}{\phi^\mu_j (\mu_j/\mu_m)^\lambda}\\
& = \frac{\mu_m^\lambda}{M\|(T^\mu_{\im Q_1})^{-1}\| \, \|\cA_{m,1}z\|_m} \sum_{j=m+1}^{q_m} \mu_j' \,\mu_j^{-1-\lambda}\\
& \ge \frac{1}{M\lambda\|(T^\mu_{\im Q_1})^{-1}\| \, \|\cA_{m,1}z\|_m} \left( 1-\frac{\mu_m^\lambda}{\mu_{q_m}^\lambda}\right)\\
& \ge \frac{1-L_1^{-\lambda}}{M\lambda\|(T^\mu_{\im Q_1})^{-1}\| \, \|\cA_{m,1}z\|_m}.
\end{split}
\]
and thus, for $m \ge n$, we get
\begin{equation}\label{eq:cA_m,nz-2}
\|\cA_{m,1}z\|_m \ge L \|\cA_{n,1} z \|_n,
\end{equation}
where
$$L=\frac{1-L_1^{-\lambda}}{M\lambda\|(T^\mu_{\im Q_1})^{-1}\|}.$$

We have, using~\eqref{prop:ESTIMATES-est-=1} and~\eqref{eq:cA_m,nz-2},
\[
\begin{split}
\frac{1}{\|\cA_{n,1} z \|_n}
& \ge \frac{1}{\|(T^\mu_{\im Q_1})^{-1}\|} \sum_{j=n+1}^\infty \frac{1}{\phi^\mu_j\|\cA_{j,1}z\|_j}\\
& \ge \frac{1}{\|(T^\mu_{\im Q_1})^{-1}\|} \sum_{j=n+1}^m \frac{1}{\phi^\mu_j\|\cA_{j,1}z\|_j}\\
& \ge \frac{L}{\|(T^\mu_{\im Q_1})^{-1}\|\|\cA_{m,1}z\|_m} \sum_{j=n+1}^m \frac{1}{\phi^\mu_j}\\
& \ge \frac{L}{\|(T^\mu_{\im Q_1})^{-1}\|\|\cA_{m,1}z\|_m} \ln \frac{\mu_{m+1}}{\mu_{n+1}}.
\end{split}
\]
Since
\[
\frac{L}{\|(T^\mu_{\im Q_1})^{-1}\|} \ln \frac{\mu_{m+1}}{\mu_{n+1}} \ge \e  \quad\quad
\Leftrightarrow \quad\quad \mu_{m+1}  \ge N_0 \mu_{n+1},
\]
where
$$N_0=\exp\{\e L^{-1} \|(T^\mu_{\im Q_1})^{-1}\|\},$$
we have
\begin{equation} \label{eq:crucial2}
\|\cA_{m,1} z \|_m \ge \e \|\cA_{n,1}z\|_n,
\end{equation}
for $m,n \in \N$ such that $\mu_m \ge N_0\mu_n$.

We let again $\gamma(n,0)=n$ and $\gamma(n,k)=q_{\gamma(n,k-1)}$ for $k\in \N$.
Proceeding like in the proof of ... we conclude that there is $K_0$ such that~\eqref{eq:crucial2} holds for all $m,n \in \N$, $m \ge n$, such that $\mu_m \ge \mu_{\gamma(n,K_0)}$.

Assume now that $m,n \in \N$, $m \ge n$ with no other restrictions and consider the largest $\ell \in \N \cup \{0\}$ such that $\mu_m \ge \mu_{\gamma(1,\ell K_0)}$. Since, for $p=0,\ldots,\ell-1$, we have
\[
\mu_{\gamma(1,(\ell-p) K_0)}\ge L_1^{K_0} \mu_{\gamma(1,(\ell-p-1) K_0)}
\ge N_0 \mu_{\gamma(1,(\ell-p-1) K_0)}
\]
we conclude, using~\eqref{eq:crucial}, that
\begin{equation}\label{eq:aux-bound-cAmn-3}
\begin{split}
& \|\cA_{m,1} x\|_m\\
& \ge L \|\cA_{\gamma(n,\ell K_0),1} x\|_{\gamma(n,\ell K_0)}\\
& = L \|\cA_{\gamma(n,\ell K_0),\gamma(n,(\ell-1) K_0)}\cA_{\gamma(n,(\ell-1) K_0),\gamma(n,(\ell-2) K_0)}
\cdots \cA_{\gamma(n,K_0),n}\cA_{n,1}x\|_{\gamma(n,\ell K_0)}\\
& \ge L\e \|\cA_{\gamma(n,(\ell-1) K_0),\gamma(n,(\ell-2) K_0)}
\cdots \cA_{n,1}x\|_{\gamma(n,\ell K_0)}\\
& \ge L \e^\ell \|\cA_{n,1} x\|_n,
\end{split}
\end{equation}
for all $m,n \in \N$, $m \ge n$, such that $\mu_m \ge \mu_{\gamma(n,K_0)}$.

Noting that
\[
\dfrac{\mu_m}{\mu_n} \le \dfrac{\mu_{\gamma(n,(\ell+1) K_0)}}{\mu_n} = \dfrac{\mu_{\gamma(n,(\ell+1) K_0)}}{\mu_{\gamma(n,(\ell+1) K_0-1)}}
\dfrac{\mu_{\gamma(n,(\ell+1) K_0-1)}}{\mu_{\gamma(n,(\ell+1) K_0-2)}}\cdots
\dfrac{\mu_{\gamma(n,1)}}{\mu_n} \le L_2^{(\ell+1) K_0}
\]
we get
\begin{equation}\label{eq:aux-bound-cAmn-4}
\ell > \dfrac{\log \left(\mu_m/\mu_n\right)}{K_0\log L_2}-1
\quad \Leftrightarrow \quad
\e^{\ell} > \e^{-1} \, \left(\mu_m/\mu_n\right)^{1/(K_0 \log L_2)}.
\end{equation}
It follows from~\eqref{eq:aux-bound-cAmn-3} and~\eqref{eq:aux-bound-cAmn-4} that, for $m \ge n$,
\[
\begin{split}
\|\cA_{m,1} x\|_m
& \ge L \e^{-1} \, \left(\mu_m/\mu_n\right)^{1/(K_0 \log L_2)} \|\cA_{n,1} x\|_n.
\end{split}
\]
Thus
\[
\begin{split}
\|\cA_{m,1} x\|_m
& \ge C\, \left(\mu_m/\mu_n\right)^b \|\cA_{n,1} x\|_n.
\end{split}
\]
where $C=L \e^{-1}$ and $b=1/(K_0 \log L_2)$.

Let now $\mu_m \ge \mu_n$, $z \in Z$ and choose $v\in Z_n$ such that $\cA_{n,1}v=z$. We have
\[
\begin{split}
\|\cA_{m,n}z\|_m  = \|\cA_{m,1}v\|_m  \ge C \left(\mu_m/\mu_n\right)^b \|\cA_{n,1} v\|_n
 \ge C \left(\mu_m/\mu_n\right)^a \|z\|_n.
\end{split}
\]
Letting $z=\cA_{n,m}x$, we get
\[
\begin{split}
\|\cA_{n,m}x\|_n \le C^{-1} \left(\mu_m/\mu_n\right)^{-b} \|x\|_m,
\end{split}
\]
and the lemma is proved by letting $D_2=C^{-1}$.
\end{proof}

We consider now the projection $P_n:X \to X_n$ associated with the decomposition $X=X_n\oplus Z_n$. It follows from the proof of Lemma~4.2 in~\cite{Mi.Ra.Sc.1998} that
\begin{equation}\label{eq:prelim-bound-Pn}
\|P_n\| \le 2/\gamma_n,
\end{equation}
where
\[
\gamma_n=\inf \left\{ \|u+v\|_n: \ (u,v) \in X_n\oplus Z_n \ \wedge \ \|u\|=\|v\|=1 \right\}.
\]
\begin{lemma}\label{lemma:bound-Pn}
We have $\displaystyle \sup_{n \in \N} \|P_n\|<+\infty$.
\end{lemma}

\begin{proof}
Let $(u,v) \in X_n\oplus Z_n$ with $\|u\|=\|v\|=1$. Using~\eqref{eq:teo:exist-dich:exponential-bound}, we
conclude that
\[
\|\cA_{m,n}(u+v)\|_m \le M(\mu_m/\mu_n)^\lambda \|u+v\|_n.
\]
Thus, using Lemma~\ref{lemma:estimate-cA-X} and Lemma~\ref{lemma:estimate-cA-Z}, we obtain, for $m \ge n$,
\[
\begin{split}
\|u+v\|_n
& \ge \frac{1}{M(\mu_m/\mu_n)^\lambda} \|\cA_{m,n}(u+v)\|_m\\
& \ge \frac{1}{M(\mu_m/\mu_n)^\lambda} \left(\|\cA_{m,n}v\|_m-\|\cA_{m,n}u\|_m\right)\\
& \ge \frac{1}{M(\mu_m/\mu_n)^\lambda} \left(D_2^{-1}(\mu_m/\mu_n)^b-D_1(\mu_m/\mu_n)^{-a}\right).
\end{split}
\]
Letting $N_0$ be chosen so that
$D_2^{-1}N_0^b-D_1N_0^{-a}>0$,
and choosing $m_0,n_0 \in \N$ such that $m_0\ge n_0$ and $\mu_{m_0} > N_0 \mu_{n_0}$, we conclude that
\[
\|u+v\|_n \ge
\frac{1}{M(\mu_{m_0}/\mu_{n_0})^\lambda} \left(D_2^{-1}N_0^b-D_1N_0^{-a}\right)>0,
\]
and the result follows from the proof of Lemma~4.2 in~\cite{Mi.Ra.Sc.1998} (see comments above).
\end{proof}
The proof of the theorem follows from the results in Lemmas~\ref{lemma:estimate-cA-X},~\ref{lemma:estimate-cA-Z} and~\ref{lemma:bound-Pn}.
\end{proof}

\section{Relation with nonuniform $\mu$-dichotomy}

We say that the sequence of linear operators $(A_m)_{m\in \N}$ (or alternatively that the linear difference equation $x_{m+1} = A_m x_m, \ m \in \N$) admits a \textit{nonuniform $\mu$-dichotomy} if there are projections $P_m$, $m\in\N$, such that $A_m|\ker P_m \to \ker P_{m+1}$ is invertible,
   $$P_m \cA_{m,n} = \cA_{m,n} P_n,\quad m, n\in\N,$$
and there are constants $\lambda, D > 0$ and $\eps>0$ such that, for every $n,m \in \N$, we have
\begin{equation}\label{eq:nonun-dich-1}
\|\cA_{m,n} P_n\| \le D \left(\mu_m/\mu_n\right)^{-\lambda} \mu_n^\eps, \quad \text{for} \ m \ge n
\end{equation}
and
\begin{equation}\label{eq:nonun-dich-2}
\|\cA_{m,n}^{-1}Q_n\| \le D \left(\mu_n/\mu_m\right)^{-\lambda} \mu_n^\eps, \quad \text{for} \ m \le n,
\end{equation}
where $Q_m=\Id-P_m$ is the complementary projection and, for $m \le n$, we use the notation $\cA_{m,n}=(\cA_{n,m})^{-1}: \ker P_n \to \ker P_m$. If additionally there are constants $K,b>0$ and $\gamma>0$ such that
\begin{equation}\label{eq:nonunif-strong-dich}
\|\cA_{m,n}\| \le K \left(\mu_n/\mu_m\right)^b \mu_n^\gamma, \quad \text{for} \ m \ge n,
\end{equation}
we say that the sequence of linear operators $(A_m)_{m\in \N}$ (or alternatively that the linear difference equation $x_{m+1} = A_m x_m, \ m \in \N$) admits a \textit{strong nonuniform $\mu$-dichotomy}

 The following result establishes an equivalence between the existence of a nonuniform $\mu$-dichotomy and the existence of a nonuniform $\mu$-dichotomy with respect to a sequence of norms satisfying some condition. This result generalizes Proposition 4.2 in~\cite{D.MN.2019}. The proof consists in adapt the arguments used in the proof of Proposition 4.2 in~\cite{D.MN.2019} to our new setting.
\begin{theorem}\label{teo:equivalence}
The following properties are equivalent:
\begin{enumerate}[1.]
\item The sequence $(A_m)_{m\in \N}$ admits a nonuniform $\mu$-dichotomy;
\item There is $C>0$ and $\eps\ge 0$ such that the sequence $(A_m)_{m\in \N}$ admits a $\mu$-dichotomy with respect to a sequence of norms $\|\cdot\|_n$ satisfying, for $x \in X$ and $m \in \N$,
\begin{equation}\label{eq:teo:equivalence-1}
\|x\| \le \|x\|_m\le C \mu_m^\eps \|x\|.
\end{equation}
\end{enumerate}
\end{theorem}

\begin{proof}
Assuming that $(A_m)_{m\in \N}$ admits a nonuniform $\mu$-dichotomy, define a sequence of norms $(\|\cdot\|_n)_{n \in \N}$ by
\begin{equation}\label{eq:seq-of-norms}
\|x\|_m := \sup_{k \ge m} (\|\cA_{k,m} P_m x\|(\mu_k/\mu_m)^\lambda)+\sup_{k \le m} (\|\cA_{k,m} Q_m x\|(\mu_m/\mu_k)^\lambda),
\end{equation}
for $x \in X$ and $n \in \N$. We have
\[
\begin{split}
\|\cA_{m,n}P_n x\|_m
& = \sup_{k \ge m} (\|\cA_{k,n} P_n x\|(\mu_k/\mu_m)^\lambda) \\
& \le \sup_{k \ge m} (D(\mu_k/\mu_n)^{-\lambda}(\mu_k/\mu_m)^\lambda \|x\|_n)\\
& = D (\mu_n/\mu_m)^\lambda \|x\|_n
\end{split}
\]
and similarly
\[
\begin{split}
\|\cA_{m,n}Q_n x\|_m
& = \sup_{k \le m} (\|\cA_{k,n}Q_n x\|(\mu_m/\mu_k)^\lambda) \\
& \le \sup_{k \ge m} (D(\mu_n/\mu_k)^{-\lambda}(\mu_m/\mu_k)^\lambda \|x\|_n)\\
& = D (\mu_m/\mu_n)^\lambda \|x\|_n.
\end{split}
\]
Additionally, we have
\[
\begin{split}
& \|x\|_m \\
& \le \sup_{k \ge m} (D(\mu_k/\mu_m)^{-\lambda}\mu_m^\eps(\mu_k/\mu_m)^\lambda\|x\|_n)+\sup_{k \le m} (D(\mu_m/\mu_k)^{-\lambda}\mu_m^\eps(\mu_m/\mu_k)^\lambda\|x\|_n\\
& =2D\mu_m^\eps \|x\|_n,
\end{split}
\]
which proves that $(A_m)_{m\in \N}$ admits a $\mu$-dichotomy with respect to the sequence of norms given by~\eqref{eq:seq-of-norms}.

Assume now that $(A_m)_{m\in \N}$ admits a $\mu$-dichotomy with respect to a sequence of norms $(\|\cdot\|_n)_{n \in \N}$ and that $\|x\| \le \|x\|_m\le C \mu_m^\eps \|x\|$. We have
\[
\|\cA_{m,n} P_n x\|\le \|\cA_{m,n} P_n x\|_n \le D (\mu_m/\mu_n)^{-\lambda} \|x\|_n \le CD (\mu_m/\mu_n)^{-\lambda}\mu_n^\eps \|x\|.
\]
and similarly
\[
\|\cA_{m,n} Q_n x\|\le \|\cA_{m,n} Q_n x\|_n \le D (\mu_n/\mu_m)^{-\lambda} \|x\|_n \le CD (\mu_n/\mu_m)^{-\lambda}\mu_n^\eps \|x\|.
\]
which proves that $(A_m)_{m\in \N}$ admits a nonuniform $\mu$-dichotomy.
\end{proof}

Next, we establish an equivalence between the existence of a strong nonuniform $\mu$-dichotomy and the existence of a nonuniform $\mu$-dichotomy with respect to a sequence of norms satisfying some condition. This result generalizes Proposition 4.3 in~\cite{D.MN.2019}. The proof consists in adapting the arguments used in the proof of Proposition 4.3 in~\cite{D.MN.2019} to our new setting.
\begin{theorem}\label{teo:equivalence-strong}
The following properties are equivalent:
\begin{enumerate}[1.]
\item The sequence $(A_m)_{m\in \N}$ admits a strong nonuniform $\mu$-dichotomy;
\item There are $C,M,a>0$ and $\eps\ge 0$ such that the sequence $(A_m)_{m\in \N}$ admits a nonuniform $\mu$-dichotomy with respect to a sequence of norms $\|\cdot\|_n$ satisfying, for $x \in X$,
\begin{equation}\label{teo:equivalence-strong-1}
\|x\| \le \|x\|_m\le C \mu_m^\eps \|x\|, \quad m \in \N
\end{equation}
and
\begin{equation}\label{teo:equivalence-strong-2}
\|\cA_{m,n}x\|_m \le M (\mu_n/\mu_m)^a \|x\|_n, \quad m \ge n.
\end{equation}
\end{enumerate}
\end{theorem}

\begin{proof}
Assume that the sequence $(A_m)_{m\in \N}$ admits a strong nonuniform $\mu$-dichotomy. We assume, without loss of generality, that the constant $\lambda$ in~\eqref{eq:nonun-dich-1} and \eqref{eq:nonun-dich-2} is less or equal than the constant $b$ in~\eqref{eq:nonunif-strong-dich}.

Letting $m=n$ in~\eqref{eq:nonun-dich-2} we conclude that $\|Q_n x\| \le D\mu_n^\eps$ for all $n \in \N$. Thus, using~\eqref{eq:nonunif-strong-dich} we have
\begin{equation}\label{eq:equivalence-strong-aux1}
\|\cA_{m,n}Q_n x\| \le K \left(\mu_n/\mu_m\right)^b \mu_n^\eps \|Q_n x\| \le KD \left(\mu_n/\mu_m\right)^b \mu_n^{2\eps},
\end{equation}
for $m \ge n$ and $x \in X$.

For $n \in \N$ and $x \in X$ define
\[
\|x\|_n^s:=\sup_{m \ge n} (\|\cA_{m,n}P_nx\|(\mu_m/\mu_n)^\lambda),
\]
\[
\|x\|_n^u:=\sup_{m \le n} (\|\cA_{m,n}Q_nx\|(\mu_n/\mu_m)^\lambda)+\sup_{m > n} (\|\cA_{m,n}Q_nx\|(\mu_m/\mu_n)^{-b})
\]
and
\[
\|x\|_n:=\|x\|_n^s+\|x\|_n^u.
\]
We obtain
\[
\|x\| \le \|P_m x\|+\|Q_m x\| \le \|P_m x\|_n^s + \|Q_m x\|_n^u = \|x\|_n
\]
and, using~\eqref{eq:nonun-dich-1},~\eqref{eq:nonun-dich-2} and~\eqref{eq:equivalence-strong-aux1}, we get
\[
\begin{split}
\|x\|_n
& = \sup_{m \ge n} (D \mu_n^\eps\|x\|)+ \sup_{m \le n} (D \mu_n^\eps\|x\|)+\sup_{m > n} (K \mu_n^\gamma \|Q_nx\|)\\
& = 2D \mu_n^\eps\|x\|+KD \mu_n^{\gamma+\eps} \|x\|\\
& \le (2+K)D \, \mu_n^{\eps+\gamma} \|x\|
\end{split}
\]
and we conclude that~\eqref{teo:equivalence-strong-1} holds with $C=(2+K)D$ and $\eps$ replaced by $\eps+\gamma$.

Let now $m\ge n$ and $x\in X$. Since $\|\cA_{m,n}P_nx\|_m^u=0$, we have, for $m \ge n$ and $x \in X$,
\[
\|\cA_{m,n}P_nx\|_m =\sup_{k \ge m} (\|\cA_{k,m}\cA_{m,n}P_nx\|(\mu_k/\mu_m)^\lambda)\le(\mu_m/\mu_n)^{-\lambda} \|x\|_n.
\]
Noting now that $\|\cA_{m,n}Q_nx\|_m^s=0$, we have, for $m \le n$ and $x \in X$,
\[
\begin{split}
\|\cA_{m,n}Q_nx\|_m
& =\sup_{k \le m} (\|\cA_{k,n}Q_nx\|(\mu_m/\mu_k)^\lambda)+\sup_{k > m} (\|\cA_{k,n}Q_nx\|(\mu_k/\mu_m)^{-b})\\
& \le 2(\mu_n/\mu_m)^{-\lambda}\|x\|_n.
\end{split}
\]
Finally, for $m \ge n$ and $x \in X$, we get
\[
\begin{split}
\|\cA_{m,n}Q_nx\|_m
& =\sup_{k \le m} (\|\cA_{k,n}Q_nx\|(\mu_m/\mu_k)^\lambda)+\sup_{k > m} (\|\cA_{k,n}Q_nx\|(\mu_k/\mu_m)^{-b})\\
& \le 2(\mu_n/\mu_m)^b\|x\|_n.
\end{split}
\]
We conclude that $(A_m)_{m \in \N}$ admits a nonuniform $\mu$-dichotomy with respect to the sequence of norms $\|\cdot\|_m$. Notice that for $m\ge n$ we have
\[
\begin{split}
\|\cA_{m,n}x\|_m
& \le \|\cA_{m,n}P_nx\|_m  + \|\cA_{m,n}Q_nx\|_m  \\
& \le (\mu_m/\mu_n)^{-\lambda} \|x\|_n  + 2(\mu_n/\mu_m)^b\|x\|_n \\
& \le 3(\mu_n/\mu_m)^b\|x\|_n.
\end{split}
\]
and we obtain~\eqref{teo:equivalence-strong-2} with $M=3$ and $a=b$.

To prove the converse statement, we assume that the sequence $(A_m)_{m\in \N}$ admits a nonuniform $\mu$-dichotomy with respect to a sequence of norms $\|\cdot\|_n$ satisfying~\eqref{teo:equivalence-strong-1} and~\eqref{teo:equivalence-strong-2}. By Theorem~\ref{teo:equivalence} we conclude that $(A_m)_{m\in \N}$ admits a nonuniform $\mu$-dichotomy. Since
$$\|\cA_{m,n}x\| \le \|\cA_{m,n}x\|_m \le M (\mu_n/\mu_m)^a \|x\|_n \le CM (\mu_n/\mu_m)^a \mu_m^\eps \|x\|$$
we conclude that the nonuniform $\mu$-dichotomy is a strong nonuniform $\mu$-dichotomy. The result is proved.
\end{proof}

\section{Parametrerized robustness of nonuniform $\mu$-dichotomies}

Let $I$ be a Banach space and, for each $\lambda \in I$, let $(B_m(\lambda))_{m \in \N}$ be a sequence of bounded linear operators acting on $X$. Assume further that, for any $n \in \N$, the function $I \ni \lambda \mapsto B_n(\lambda)$ is continuous.

\begin{theorem}\label{teo:robustness-dich-bounds}
Let $\mu$ be a discrete growth rate and assume that $(A_m)_{m \in \N}$ admits a $\mu$-dichotomy with respect to the sequence of norms $(\|\cdot\|_m)_{m \in \N}$. Moreover, assume that there is an increasing sequence $(q_n)$ with $q_n\ge n+1$, $n \in \N$, and constants $L_1,L_2>1$ such that, for all $n \in\ \N$,
\begin{equation}\label{teo:robustness-dich-bounds-a1}
L_1 \le \frac{\mu_{q_n}}{\mu_n} \le L_2,
\end{equation}
and that there are constants $\lambda,M>0$ such that, for all $m \ge n$ and $x \in X$,
\begin{equation}\label{teo:robustness-dich-bounds-a2}
\|\cA_{m,n} x\|_m \le M (\mu_m/\mu_n)^\lambda\|x\|_n.
\end{equation}
Assume further that, for all $m \in \N$, $x \in X$ and $\lambda,\mu \in I$, there is $c>0$ such that
\begin{equation}\label{eq:robustness-dich-bounds-1}
\|\phi_{m+1}^\mu B_m(\lambda)x\|_{m+1} \le c\|x\|_m.
\end{equation}
Then, if $c>0$ is sufficiently small, the sequence $(A_m+B_m(\lambda))_{m \in \N}$ admits a nonuniform $\mu$-dichotomy with respect to the sequence of norms $\|\cdot\|_m$ for each $\lambda \in I$. Additionally, if there is $d>0$ such that, for all $m \in \N$, $x \in X$ and $\lambda,\mu \in I$, we have
\begin{equation}\label{eq:robustness-dich-bounds-2}
\|\phi_{m+1}^\mu \left(B_m(\lambda)-B_m(\mu)\right)x\|_{m+1} \le d\|\lambda-\mu\| \, \|x\|_m,
\end{equation}
we can choose the projections $P_{m,\lambda}$ in such a way that $\lambda \mapsto P_{m,\lambda}$ is locally Lipschitz.
\end{theorem}

\begin{proof}
By Theorem~\ref{teo:invertibility} the operator $T_Z^\mu$ corresponding to $(A_m)_{m \in \N}$ is invertible. Let $T_{Z,\lambda}^\mu$ be the operator in~\eqref{eq:def-TmuZ} correponding to the sequence of bounded linear operators $\left(A_m+B_m(\mu)\right)_{m \in \N}$. Since $(T^\mu_{Z,\lambda})_0 -(T^\mu_Z)_0=0-0=0$ and, for $m \ge 1$,
\[
\begin{split}
(T^\mu_{Z,\lambda})_m -(T^\mu_Z)_m
& =\phi^\mu_m(x_m-(A_{m-1}+B_{m-1}(\lambda))x_{m-1}-x_m+A_{m-1}x_{m-1})\\
& =\phi^\mu_m B_{m-1}(\lambda)x_{m-1},
\end{split}
\]
we conclude that, for each $x \in Y_Z$ and $\lambda \in I$,
\begin{equation}\label{eq:follow-from}
\begin{split}
\|(T^\mu_{Z,\lambda}-T^\mu_Z)x\|_\infty
& =\sup_{m\ge 2} \|\phi^\mu_m B_{m-1}(\lambda)x_{m-1}\|_m \\
& \le c \sup_{m\ge 2} \|x_{m-1}\|_{m-1} = c\|x\|_\infty \le c\|x\|_{T^\mu_Z}.
\end{split}
\end{equation}
Thus, for each $x \in Y_Z$ and $\lambda \in I$,
\begin{equation}
\|T^\mu_{Z,\lambda}x\|_\infty \le \|T^\mu_Z\|_\infty+\|(T^\mu_{Z,\lambda}-T^\mu_Z)x\|_\infty \le c\|x\|_{T^\mu_Z}.
\end{equation}
We conclude that, for each $\lambda \in I$, the domain of the operator $T^\mu_{Z,\lambda}$ is $\cD(T^\mu_Z)$ and that
$T^\mu_{Z,\lambda}:(\cD(T^\mu_Z),\|\cdot\|_{T^\mu_Z})\to Y_0$ is a bounded linear operator. Additionally, it follows from~\eqref{eq:follow-from}
that $T^\mu_{Z,\lambda}$ is invertible whenever $c>0$ is sufficiently small.

Denote by $(\cC_{m,n}^\lambda)_{(m,n)\in \Delta}$ the discrete evolution family associated with the sequence $(A_m+B_m(\lambda))_{m\in \N}$, given by
\[
\cC_{m,n}^\lambda=
\begin{cases}
(A_{m-1}+B_{m-1}(\lambda))\cdots (A_n+B_n(\lambda)), & m>n\\
Id, & m=n
\end{cases}.
\]
The following auxiliary result establishes a bound for the growth of $\|\cC_{m,n}^\lambda x\|$.
\begin{lemma}\label{lemma:bound-growth-norm-cC}
Let $0<\delta<1$ and assume that $M>1$ and that the constant $c>0$ in~\eqref{eq:robustness-dich-bounds-1} is sufficiently small so that $N:=M/(1-cM/\delta)>1$. Then, for all $k \ge n$, the following holds:
\begin{equation}\label{eq:bound-growth-norm-cC}
\|\cC_{k,n}^\lambda x\|_k \le N \left(\mu_k/\mu_n\right)^{\lambda+\delta} \|x\|_n.
\end{equation}
\end{lemma}
\begin{proof}
Let $n \in \N$ and $0<\delta<1$. We will proceed by induction on $m \ge n$.

For $k=n$ we have
\[
\|\cC_{n,n}^\lambda x\|_n = \|x\|_n  = \left(\mu_n/\mu_n\right)^{\lambda+\delta} \|x\|_n \le N \left(\mu_n/\mu_n\right)^{\lambda+\delta} \|x\|_n,
\]
and we have~\eqref{eq:bound-growth-norm-cC} with $m=n$.

Assume now that~\eqref{eq:bound-growth-norm-cC} holds for all $k=n, \ldots m-1$. Using~\eqref{teo:robustness-dich-bounds-a2} and \eqref{eq:robustness-dich-bounds-1} we get
\[
\begin{split}
\|\cC_{m,n}^\lambda x\|_m
&  \le \|\cA_{m,n}x\|_m+\sum_{k=n}^{m-1} \|\cA_{m,k+1} B_k(\lambda) \, \cC^\lambda_{k,n}x\|_m\\
& \le \|\cA_{m,n}x\|_m+\sum_{k=n}^{m-1} M \left(\frac{\mu_m}{\mu_{k+1}}\right)^\lambda \|B_k(\lambda) \, \cC^\lambda_{k,n}x\|_{k+1}\\
& \le M \left(\frac{\mu_m}{\mu_n}\right)^\lambda \|x\|_n +\sum_{k=n}^{m-1} cM \left(\frac{\mu_m}{\mu_{k+1}}\right)^\lambda \frac{1}{\phi_{k+1}^\mu}\|\cC^\lambda_{k,n}x\|_k
\end{split}
\]
Thus, taking into account the assumption $0<\delta<1$ and the definition of $N$, we obtain
\[
\begin{split}
\|\cC_{m,n}^\lambda x\|_m
& \le M \left(\frac{\mu_m}{\mu_n}\right)^\lambda \|x\|_n +cMN\|x\|_n\sum_{k=n}^{m-1} \left(\frac{\mu_m}{\mu_{k+1}}\right)^\lambda \frac{\mu_{k+1}'}{\mu_{k+1}}
\left(\frac{\mu_k}{\mu_n}\right)^{\lambda+\delta}\\
& \le M \left(\frac{\mu_m}{\mu_n}\right)^\lambda \|x\|_n +cMN\left(\frac{\mu_m}{\mu_n}\right)^{\lambda} \, \mu_n^{-\delta}\|x\|_n\sum_{k=n}^{m-1} \mu_{k+1}' \mu_{k+1}^{-(1-\delta)}\\
& \le M \left(\frac{\mu_m}{\mu_n}\right)^\lambda \|x\|_n +\frac{c}{\delta}MN\left(\frac{\mu_m}{\mu_n}\right)^{\lambda} \|x\|_n \left(\left(\frac{\mu_m}{\mu_n}\right)^\delta-1\right)\\
& \le M(1+\frac{c}{\delta}N)\left(\frac{\mu_m}{\mu_n}\right)^{\lambda+\delta} \|x\|_n\\
& =N\left(\frac{\mu_m}{\mu_n}\right)^{\lambda+\delta} \|x\|_n.
\end{split}
\]
and~\eqref{eq:bound-growth-norm-cC} holds for $k=m$. The lemma is proved.
\end{proof}

By Theorem~\ref{teo:exist-dich}, noting that~\eqref{eq:bound-growth-norm-cC} corresponds to assumption~\eqref{eq:teo:exist-dich:exponential-bound} and that we have invertibility of $T_{Z,\lambda}^\mu$, we conclude that $(\cC_{m,n}^\lambda)_{(m,n)\in \Delta}$ admits a nonuniform $\mu$-dichotomy with respect to the sequence of norms $\|\cdot\|_m$ as long as the constant $c>0$ in~\eqref{eq:robustness-dich-bounds-1} is sufficiently small.

We note now that the projections $P_{m,\lambda}$ associated with the dichotomy considered for the evolution family $\cC_{m,n}^\lambda$ can be chosen so that
\[
P_{1,\lambda}v=v-((T^\mu_{Z,\lambda})^{-1} w^1)_1 \quad \text{ and } \quad P_{m,\lambda}v=((T^\mu_{Z,\lambda})^{-1} w^m)_m, \quad m>1
\]
where $w^1=(0,-\phi_2^\mu A_1v,0,0,\ldots)$ and $w^n$ are defined as in the proof of Lemma~\ref{lemma:bound-Pn}. We will consider these projections in our proof.

We have, for all $x \in \cD(T_Z^\mu)$ and $\lambda_1,\lambda_2 \in I$,
\begin{equation}\label{eq:^Tmu-e-Lip}
\begin{split}
\|(T^\mu_{Z,\lambda_1}-T^\mu_{Z,\lambda_2})x\|_\infty
& = \sup_{m\ge 2} \|\phi^\mu_m \left[B_{m-1}(\lambda)-B_{m-1}(\mu)\right]x_{m-1}\|_m \\
& \le d\|\lambda_1-\lambda_2\| \, \sup_{m\ge 2} \|x_{m-1}\|_{m-1} \\
& = d\|\lambda_1-\lambda_2\| \|x\|_\infty
\le d\|\lambda_1-\lambda_2\| \|x\|_{T^\mu_Z}.
\end{split}
\end{equation}
We conclude that the function $I \in \lambda \mapsto T^\mu_{Z,\lambda}$ is Lipschitz.

Using~\eqref{eq:^Tmu-e-Lip}, we have, for $m \in \N$,
\begin{equation}\label{eq:Plambda1-Plambda2}
\begin{split}
\|P_{m,\lambda_1}v-P_{m,\lambda_2}v\|_m
& = \left\|\left((T^\mu_{Z,\lambda_1}-T^\mu_{Z,\lambda_2})w^m\right)_m\right\|_m\\
& \le \left\|(T^\mu_{Z,\lambda_1}-T^\mu_{Z,\lambda_2})w^m\right\|_\infty\\
& \le \left\|(T^\mu_{Z,\lambda_1}-T^\mu_{Z,\lambda_2})w^m\right\|_{T^\mu_z}\\
& \le \|T^\mu_{Z,\lambda_1}-T^\mu_{Z,\lambda_2})\|\,\|w^m\|_\infty.
\end{split}
\end{equation}
Thus we conclude that
\begin{equation}\label{eq:Plambda1-Plambda2a}
\|P_{1,\lambda_1}v-P_{1,\lambda_2}v\|_1 \le \|T^\mu_{Z,\lambda_1}-T^\mu_{Z,\lambda_2})\|\,\|\phi_2^\mu A_1v\|_1
\end{equation}
and that, for $m > 1$,
\begin{equation}\label{eq:Plambda1-Plambda2b}
\|P_{m,\lambda_1}v-P_{m,\lambda_2}v\|_m \le \|T^\mu_{Z,\lambda_1}-T^\mu_{Z,\lambda_2})\|\,\|v\|_m.
\end{equation}

Noting, for all $\lambda_1,\lambda_2 \in I$, we have
\[
\begin{split}
\|(T^\mu_{Z,\lambda_1})^{-1}-(T^\mu_{Z,\lambda_2})^{-1}\|
& =\|(T^\mu_{Z,\lambda_1})^{-1}T^\mu_{Z,\lambda_2} (T^\mu_{Z,\lambda_2})^{-1}-(T^\mu_{Z,\lambda_1})^{-1}T^\mu_{Z,\lambda_1} (T^\mu_{Z,\lambda_2})^{-1}\| \\
& \le \|(T^\mu_{Z,\lambda_1})^{-1}\|\,\|(T^\mu_{Z,\lambda_2})^{-1}\|\|T^\mu_{Z,\lambda_1}-T^\mu_{Z,\lambda_2}\|
\end{split}
\]
and that
\[
\begin{split}
\|(T^\mu_{Z,\lambda_1})^{-1}\|
& =\|(T^\mu_{Z,\lambda_1})^{-1}-(T^\mu_{Z,\lambda_2})^{-1}+(T^\mu_{Z,\lambda_2})^{-1}\|\\
& \le \|(T^\mu_{Z,\lambda_1})^{-1}-(T^\mu_{Z,\lambda_2})^{-1}\|+\|(T^\mu_{Z,\lambda_2})^{-1}\|
\end{split}
\]
we conclude that,
\begin{equation}\label{bound-Tz-1-Tz-1}
\|(T^\mu_{Z,\lambda_1})^{-1}-(T^\mu_{Z,\lambda_2})^{-1}\|
= \frac{\|(T^\mu_{Z,\lambda_1})^{-1}\|^2\|T^\mu_{Z,\lambda_1}-T^\mu_{Z,\lambda_2}\|}
{1-\|(T^\mu_{Z,\lambda_1})^{-1}\|\|T^\mu_{Z,\lambda_1}-T^\mu_{Z,\lambda_2}\|}.
\end{equation}
Thus, by~\eqref{eq:Plambda1-Plambda2b}, \eqref{bound-Tz-1-Tz-1} and~\eqref{eq:^Tmu-e-Lip}, we obtain
\[
\begin{split}
\|P_{m,\lambda_1}v-P_{m,\lambda_2}v\|_m
& \le \frac{\|(T^\mu_{Z,\lambda_1})^{-1}\|^2\|T^\mu_{Z,\lambda_1}-T^\mu_{Z,\lambda_2}\|}
{1-\|(T^\mu_{Z,\lambda_1})^{-1}\|\|T^\mu_{Z,\lambda_1}-T^\mu_{Z,\lambda_2}\|}\,\|v\|_m\\
& \le \frac{d\|(T^\mu_{Z,\lambda_1})^{-1}\|^2\|v\|_m}
{1-\|(T^\mu_{Z,\lambda_1})^{-1}\|\|T^\mu_{Z,\lambda_1}-T^\mu_{Z,\lambda_2}\|}\,\|\lambda_1-\lambda_2\|
\end{split}
\]
for $m>1$ and, by~\eqref{eq:Plambda1-Plambda2a}, \eqref{bound-Tz-1-Tz-1} and~\eqref{eq:^Tmu-e-Lip}, we have
\[
\begin{split}
\|P_{1,\lambda_1}v-P_{1,\lambda_2}v\|_1
& \le \frac{d\phi^\mu_2 \|(T^\mu_{Z,\lambda_1})^{-1}\|^2 1\|A_1\|v\|_1}
{1-\|(T^\mu_{Z,\lambda_1})^{-1}\|\|T^\mu_{Z,\lambda_1}-T^\mu_{Z,\lambda_2}\|}\,\|\lambda_1-\lambda_2\|
\end{split}
\]
and we conclude that, for each $m \in \N$, the function $I \in \lambda \mapsto P_{m,\lambda}$ is locally Lipschitz.
\end{proof}

\begin{theorem}\label{teo:robustness-strong-dich-bounds}
Let $\mu$ be a discrete growth rate and assume that $(A_m)_{m \in \N}$ admits a strong nonuniform $\mu$-dichotomy. Moreover, assume that there is an increasing sequence $(q_n)$ with $q_n\ge n+1$, $n \in \N$, and constants $L_1,L_2>1$ such that, for all $n \in\ \N$,
\begin{equation}\label{teo:robustness-strong-dich-bounds-a1}
L_1 \le \frac{\mu_{q_n}}{\mu_n} \le L_2.
\end{equation}
Assume further that, for all $m \in \N$ and $\lambda,\mu \in I$, there is $c>0$ such that
\begin{equation}\label{eq:robustness-strong-dich-bounds-1}
\|\mu_{m+1}^\eps \phi_{m+1}^\mu B_m(\lambda)\| \le c,
\end{equation}
where $\eps>0$ is the exponent of $\mu_n$ the definition of strong nonuniform $\mu$-dichotomy.
Then, if $c>0$ is sufficiently small, the sequence $(A_m+B_m(\lambda))_{m \in \N}$ admits a strong nonuniform $\mu$-dichotomy for each $\lambda \in I$. Additionally, if there is $d>0$ such that
\begin{equation}\label{eq:robustness-strong-dich-bounds-2}
\|\mu_{m+1}^\eps\phi_{m+1}^\mu \left(B_m(\lambda)-B_m(\mu)\right)\| \le d|\lambda-\mu|,
\end{equation}
for all $m \in \N$ and $\lambda,\mu \in I$, we can choose the projections $P_{m,\lambda}$ in such a way that $\lambda \mapsto P_{m,\lambda}$ is locally Lipschitz.
\end{theorem}

\begin{proof}
Since $(A_m)_{m \in \N}$ admits a strong nonuniform $\mu$-dichotomy, Theorem~\ref{teo:equivalence-strong} shows that there are $C,M,a>0$ and $\eps\ge 0$ such that the sequence $(A_m)_{m\in \N}$ admits a nonuniform $\mu$-dichotomy with respect to a sequence of norms $\|\cdot\|_n$ satisfying, for $x \in X$,
\begin{equation}\label{eq:aux-equivalencia-normas}
\|x\| \le \|x\|_m\le C \mu_m^\eps \|x\|, \quad m \in \N
\end{equation}
and
$$\|\cA_{m,n}x\|_m \le M (\mu_n/\mu_m)^a \|x\|_n, \quad m \ge n.$$
Note that the constant $\eps>0$ can be taken to be the exponent $\eps$ in the definition of nonuniform $\mu$-dichotomy. We will consider in what follows that $\eps>0$ is that exponent.

Using~\eqref{eq:aux-equivalencia-normas} and~\eqref{eq:robustness-strong-dich-bounds-1}, we have
\begin{equation}\label{eq:robustness-cond-strong-dich-from-dich}
\begin{split}
\|\phi_{m+1}^\mu B_m(\lambda)x\|_{m+1}
& \le C\mu_{m+1}^\eps \|\phi_{m+1}^\mu B_m(\lambda) x\|\\
& = C \|\mu_{m+1}^\eps \phi_{m+1}^\mu B_m(\lambda) x\|\\
& \le cC \|x\| \\
& \le cC \|x\|_m,
\end{split}
\end{equation}
and we conclude that assumption~\eqref{eq:robustness-dich-bounds-1} in Theorem~\ref{teo:robustness-dich-bounds} is fulfilled in the  present context.
Thus, Theorem~\ref{teo:robustness-dich-bounds} shows that for sufficiently small $c>0$ the sequence $(A_m+B_m(\lambda))_{m \in \N}$ admits a nonuniform $\mu$-dichotomy with respect to the sequence of norms $\|\cdot\|_m$ for each $\lambda \in I$.

Using~\eqref{eq:aux-equivalencia-normas} we obtain
\[
\|\cC^\lambda_{m,n}x\|\le \|\cC^\lambda_{m,n}x\|_m \le N\left(\frac{\mu_k}{\mu_m}\right)^{\lambda+\delta} \|x\|_n \le
NC \left(\frac{\mu_k}{\mu_m}\right)^{\lambda+\delta} \mu_m^\eps \|x\|
\]
and thus $\|\cC^\lambda_{m,n}\|\le D \left(\frac{\mu_k}{\mu_m}\right)^b \mu_m^\eps$ with $D=NC$ and $b=\lambda+\delta$. We conclude that the sequence $(A_m+B_m(\lambda))_{m \in \N}$ admits a strong nonuniform $\mu$-dichotomy.

Finally, using~\eqref{eq:aux-equivalencia-normas}, we note that
\[
\begin{split}
\|\phi_{m+1}^\mu \left(B_m(\lambda)-B_m(\mu)\right)x\|_{m+1}
& \le C \mu_{m+1}^\eps \|\phi_{m+1}^\mu \left(B_m(\lambda)-B_m(\mu)\right)\| \|x\|\\
& \le C \|\mu_{m+1}^\eps\phi_{m+1}^\mu \left(B_m(\lambda)-B_m(\mu)\right)\| \|x\|\\
& \le Cd \|\lambda-\mu\| \|x\|\\
& \le Cd \|\lambda-\mu\| \|x\|_m
\end{split}
\]
and we conclude that assumption~\eqref{eq:robustness-dich-bounds-2} in Theorem~\ref{teo:robustness-dich-bounds} holds in the  present context. We conclude that we can choose the projections $P_{m,\lambda}$ in such a way that $\lambda \mapsto P_{m,\lambda}$ is locally Lipschitz.
The theorem follows.
\end{proof}

\section{Examples}

In this section we will apply our results to three particular families of growth rates: exponential, polynomial and logarithmic.
For polynomial and exponential growth rates we can see that our results include, as particular cases, existent results in the literature; to the best of our knowledge, the case of logarithmic growth rates is considered here for the first time.

\subsection{Exponential growth rate}

We begin with the case of the exponential growth rate $\eta=(\e^n)_{n \in \N}$. For this growth rate, condition~\eqref{eq:COND-minimal-growth} is satisfied with $q_n=n+1$ and $L_1=L_2=\e$ (in this case we get $\mu_{q_n}/\mu_n=\mu_{n+1}/\mu_n = \e$).  In the nonuniform exponential context, we obtain $\phi^\eta_n=(e-1)^{-1}$ and the linear operator $T^\eta_Z:(\cD(T^\eta_Z),\|\cdot\|_{T^\eta_Z}) \to Y_0$, is given by
$$(T^\eta_Z x)_1=0 \quad \text{ \and } \quad (T^\eta_Z x)_m=(\e-1)^{-1}(x_m-A_{m-1}x_{m-1}), \quad \ m \ge 1.$$
We have the following corollaries of Theorems~\ref{teo:invertibility} and~\ref{teo:exist-dich}. In corollaries~\ref{corollary-exp-1} and~\ref{corollary-exp-2} below we recover, respectively, Theorems 1 and 2 in~\cite{B.D.V.2016}.
\begin{corollary}[of Theorem~\ref{teo:invertibility}]\label{corollary-exp-1}
If the sequence $(A_m)_{m \in \N}$ admits a nonuniform exponential dichotomy with respect to a sequence of norms $(\|\cdot\|_m)_{m \in \N}$, then the operator $T^\mu_{\im Q_1}$ is invertible.
\end{corollary}
\begin{corollary}[of Theorem~\ref{teo:exist-dich}]\label{corollary-exp-2}
If $T^\eta_Z$ is an invertible operator for some closed subespace $Z \subset X$ and there are constants $\lambda,M>0$ such that
$$
\|\cA_{m,n}x\|_m\le M\e^{\lambda(m-n)}\|x\|_n,
$$
for all $m \ge n$ and $x \in X$, then $(A_m)_{m \in \N}$ admits a nonuniform exponential dichotomy with respect to the sequence of norms $(\|\cdot\|_m)_{m \in \N}$.
\end{corollary}

Note that it is immediate that Corollaries~\ref{cor:exponential1} and~\ref{cor:exponential2} still hold if the linear operator $T^\eta_Z$ is replaced by the linear operator
$\tilde{T}^\eta_Z:(\cD(T^\eta_Z),\|\cdot\|_{\tilde{T}^\eta_Z}) \to Y_0$, given by
$$(\tilde{T}^\eta_Z x)_1=0 \quad \text{ \and } \quad (\tilde{T}^\eta_Z x)_m=x_m-A_{m-1}x_{m-1}, \quad \ m \ge 1,$$
and thus our results correspond to Theorems --- in~\cite{B.D.V.AH.2018}.

Theorem~\ref{teo:equivalence} becomes in the nonuniformly exponential context:

\begin{corollary}[of Theorem~\ref{teo:equivalence}]
The following properties are equivalent:
\begin{enumerate}[1.]
\item The sequence $(A_m)_{m\in \N}$ admits a nonuniform  exponential dichotomy;
\item There is $C>0$ and $\eps\ge 0$ such that the sequence $(A_m)_{m\in \N}$ admits a nonuniform exponential dichotomy with respect to a sequence of norms $\|\cdot\|_n$ satisfying, for $x \in X$ and $m \in \N$,
$$\|x\| \le \|x\|_m\le C \e^{\eps m} \|x\|.$$
\end{enumerate}
\end{corollary}

\begin{corollary}[of Theorem~\ref{teo:robustness-dich-bounds}]
Assume that $(A_m)_{m \in \N}$ admits a nonuniform exponential dichotomy with respect to the sequence of norms $(\|\cdot\|_m)_{m \in \N}$ and that there are constants $\lambda,M>0$ such that, for all $m \ge n$ and $x \in X$,
$$\|\cA_{m,n} x\|_m \le M \e^{\lambda(m-n)}\|x\|_n.$$
Assume further that, for all $m \in \N$, $x \in X$ and $\lambda,\mu \in I$, there is $c>0$ such that
$$\|B_m(\lambda)x\|_{m+1} \le c\|x\|_m.$$
Then, if $c>0$ is sufficiently small, the sequence $(A_m+B_m(\lambda))_{m \in \N}$ admits a nonuniform $\mu$-dichotomy with respect to the sequence of norms $\|\cdot\|_m$ for each $\lambda \in I$. Additionally, if there is $d>0$ such that, for all $m \in \N$, $x \in X$ and $\lambda,\mu \in I$, we have
$$\|\left(B_m(\lambda)-B_m(\mu)\right)x\|_{m+1} \le d\|\lambda-\mu\| \, \|x\|_m,$$
we can choose the projections $P_{m,\lambda}$ in such a way that $\lambda \mapsto P_{m,\lambda}$ is locally Lipschitz.
\end{corollary}

\begin{corollary}[of Theorem~\ref{teo:robustness-strong-dich-bounds}]
Assume that $(A_m)_{m \in \N}$ admits a strong nonuniform exponential dichotomy and that, for all $m \in \N$ and $\lambda,\mu \in I$, there is $c>0$ such that
$$\|B_m(\lambda)\| \le c\e^{-\eps m},$$
where $\eps>0$ is the exponent in the definition of strong nonuniform exponential dichotomy.
Then, if $c>0$ is sufficiently small, the sequence $(A_m+B_m(\lambda))_{m \in \N}$ admits a strong nonuniform exponential dichotomy for each $\lambda \in I$. Additionally, if there is $d>0$ such that
$$
\|B_m(\lambda)-B_m(\mu)\| \le d\e^{-m\eps}|\lambda-\mu|,
$$
for all $m \in \N$ and $\lambda,\mu \in I$, we can choose the projections $P_{m,\lambda}$ in such a way that $\lambda \mapsto P_{m,\lambda}$ is locally Lipschitz.
\end{corollary}

\subsection{Polynomial growth rate}

Now we consider the case of polynomial growth rates: $p=(n)_{n \in \N}$. This case was already discussed by Dragi\v cevi\'c in~\cite{D.MN.2019} and his work was the departure point for several of our results. In this case, condition~\eqref{eq:COND-minimal-growth} is satisfied, for instance, with $q_n=2n+1$, $L_1=2$ and $L_2=3$ (in this case we get $\mu_{q_n}/\mu_n =(2n+1)/n = 2+1/n$).  In the nonuniform polynomial context, we have $\phi^\eta_n=n$ and the linear operator $T^p_Z:(\cD(T^p_Z),\|\cdot\|_{T^p_Z}) \to Y_0$, is given by
$$(T^p_Z x)_1=0 \quad \text{ \and } \quad (T^p_Z x)_m=n(x_m-A_{m-1}x_{m-1}), \quad \ m \ge 1$$
and coincides with the operator $T_Z$ introduced in~\cite{D.MN.2019}. We have the following corollaries of Theorems~\ref{teo:invertibility} and~\ref{teo:exist-dich} in the context of nonuniform polynomial behavior that coincide with Theorems 2 and 3 in~\cite{D.MN.2019}:
\begin{corollary}[of Theorem~\ref{teo:invertibility}]
If the sequence $(A_m)_{m \in \N}$ admits a nonuniform polynomial dichotomy with respect to a sequence of norms $(\|\cdot\|_m)_{m \in \N}$, then the operator $T^\mu_{\im Q_1}$ is invertible.
\end{corollary}
and
\begin{corollary}[of Theorem~\ref{teo:exist-dich}]
If $T^\eta_Z$ is an invertible operator for a closed subespace $Z \subset X$ and there are constants $\lambda,M>0$ such that
$$
\|\cA_{m,n}x\|_m\le M(m/n)^\lambda\|x\|_n,
$$
for all $m \ge n$ and $x \in X$, then $(A_m)_{m \in \N}$ admits a nonuniform polynomial dichotomy with respect to the
sequence of norms $(\|\cdot\|_m)_{m \in \N}$.
\end{corollary}

Theorem~\ref{teo:equivalence} becomes in the nonuniformly polynomial context:

\begin{corollary}[of Theorem~\ref{teo:equivalence}]
The following properties are equivalent:
\begin{enumerate}[1.]
\item The sequence $(A_m)_{m\in \N}$ admits a nonuniform polynomial dichotomy;
\item There is $C>0$ and $\eps\ge 0$ such that the sequence $(A_m)_{m\in \N}$ admits a nonuniform polynomial dichotomy with respect to a sequence of norms $\|\cdot\|_n$ satisfying, for $x \in X$ and $m \in \N$,
$$\|x\| \le \|x\|_m\le C m^\eps \|x\|.$$
\end{enumerate}
\end{corollary}

\begin{corollary}[of Theorem~\ref{teo:robustness-dich-bounds}]
Assume that $(A_m)_{m \in \N}$ admits a nonuniform polynomial dichotomy with respect to the sequence of norms $(\|\cdot\|_m)_{m \in \N}$ and that there are constants $\lambda,M>0$ such that, for all $m \ge n$ and $x \in X$,
$$
\|\cA_{m,n} x\|_m \le M (m/n)^\lambda \|x\|_n.
$$
Assume further that, for all $m \in \N$, $x \in X$ and $\lambda,\mu \in I$, there is $c>0$ such that
$$
\| B_m(\lambda)x\|_{m+1} \le \frac{c}{m+1} \, \|x\|_m.
$$
Then, if $c>0$ is sufficiently small, the sequence $(A_m+B_m(\lambda))_{m \in \N}$ admits a nonuniform polynomial dichotomy with respect to the sequence of norms $\|\cdot\|_m$ for each $\lambda \in I$. Additionally, if there is $d>0$ such that, for all $m \in \N$, $x \in X$ and $\lambda,\mu \in I$, we have
$$
\|\left(B_m(\lambda)-B_m(\mu)\right)x\|_{m+1} \le \frac{d}{m+1}\|\lambda-\mu\| \, \|x\|_m,
$$
we can choose the projections $P_{m,\lambda}$ in such a way that $\lambda \mapsto P_{m,\lambda}$ is locally Lipschitz.
\end{corollary}

The following result corresponds to Theorem 5.1 in~\cite{D.MN.2019}. In fact, the result below generalizes Theorem 5.1 in~\cite{D.MN.2019} since the condition~\eqref{eq:robustness-strong-poly-dich-bounds-1} assumed here is weaker than the condition considered in that paper (that reads in our notation $\|B_m(\lambda)\| \le c/((m+1)^{\eps+2})$).

\begin{corollary}[of Theorem~\ref{teo:robustness-strong-dich-bounds}]
Assume that $(A_m)_{m \in \N}$ admits a strong nonuniform polynomial dichotomy and that, for all $m \in \N$ and $\lambda,\mu \in I$, there is $c>0$ such that
$$
\|B_m(\lambda)\| \le \frac{c}{(m+1)^{\eps+1}},
$$
where $\eps>0$ is the exponent in the definition of strong nonuniform $\mu$-dichotomy.
Then, if $c>0$ is sufficiently small, the sequence $(A_m+B_m(\lambda))_{m \in \N}$ admits a strong nonuniform polynomial dichotomy for each $\lambda \in I$. Additionally, if there is $d>0$ such that
$$\|\left(B_m(\lambda)-B_m(\mu)\right)\| \le \frac{d}{(m+1)^{\eps+1}}|\lambda-\mu|,$$
for all $m \in \N$ and $\lambda,\mu \in I$, we can choose the projections $P_{m,\lambda}$ in such a way that $\lambda \mapsto P_{m,\lambda}$ is locally Lipschitz.
\end{corollary}

\subsection{Logarithmic growth rate}

In our last family of examples we consider the case of logarithmic growth rates: $\nu=(\log(n+1))_{n \in \N}$. As far as we are aware, this case is discussed for the first time in the present work. For these growth rates, condition~\eqref{eq:COND-minimal-growth} is satisfied, for instance, with $q_n=(n+1)^2$, $L_1=2$ and $L_2=2+\log(5/4)/\log 2$ (in this case we get $\mu_{q_n}/\mu_n =\log((n+1)^2+1)/\log(n+1)$).  In this nonuniform logarithmic context, we have $\phi^\nu_n=\log(n+1)/\log(1+1/(n+1))$ and the linear operator $T^\nu_Z:(\cD(T^\nu_Z),\|\cdot\|_{T^\nu_Z}) \to Y_0$, is given by
$$(T^\nu_Z x)_1=0 \ \text{ \and } \ (T^\nu_Z x)_m=\log(n+1)/\log(1+1/(n+1))(x_m-A_{m-1}x_{m-1}), \ \ m \ge 1$$
We have the following corollaries of Theorems~\ref{teo:invertibility} and~\ref{teo:exist-dich} in this context:
\begin{corollary}[of Theorem~\ref{teo:invertibility}]
If the sequence $(A_m)_{m \in \N}$ admits a nonuniform logarithmic dichotomy with respect to a sequence of norms $(\|\cdot\|_m)_{m \in \N}$, then the operator $T^\nu_{\im Q_1}$ is invertible.
\end{corollary}
and
\begin{corollary}[of Theorem~\ref{teo:exist-dich}]
If $T^\nu_Z$ is an invertible operator for a closed subespace $Z \subset X$ and there are constants $\lambda,M>0$ such that
$$
\|\cA_{m,n}x\|_m\le M\left(\frac{\log(m+2)}{\log(n+2)}\right)^\lambda\|x\|_n,
$$
for all $m \ge n$ and $x \in X$, then $(A_m)_{m \in \N}$ admits a nonuniform logarithmic dichotomy with respect to the
sequence of norms $(\|\cdot\|_m)_{m \in \N}$.
\end{corollary}

Theorem~\ref{teo:equivalence} becomes in the nonuniformly logarithmic context:

\begin{corollary}[of Theorem~\ref{teo:equivalence}]
The following properties are equivalent:
\begin{enumerate}[1.]
\item The sequence $(A_m)_{m\in \N}$ admits a nonuniform logarithmic dichotomy;
\item There is $C>0$ and $\eps\ge 0$ such that the sequence $(A_m)_{m\in \N}$ admits a nonuniform logarithmic dichotomy with respect to a sequence of norms $\|\cdot\|_n$ satisfying, for $x \in X$ and $m \in \N$,
$$\|x\| \le \|x\|_m\le C (\log(m+2))^\eps \|x\|.$$
\end{enumerate}
\end{corollary}

\begin{corollary}[of Theorem~\ref{teo:robustness-dich-bounds}]
Assume that $(A_m)_{m \in \N}$ admits a nonuniform logarithmic dichotomy with respect to the sequence of norms $(\|\cdot\|_m)_{m \in \N}$ and that there are constants $\lambda,M>0$ such that, for all $m \ge n$ and $x \in X$,
$$\|\cA_{m,n} x\|_m \le M (\ln(m+2)/\ln(n+2))^\lambda\|x\|_n.$$
Assume further that, for all $m \in \N$, $x \in X$ and $\lambda,\mu \in I$, there is $c>0$ such that
$$\|B_m(\lambda)x\|_{m+1} \le c\frac{\ln((m+2)/(m+1))}{\ln(m+2)}\|x\|_m.$$
Then, if $c>0$ is sufficiently small, the sequence $(A_m+B_m(\lambda))_{m \in \N}$ admits a nonuniform logarithmic dichotomy with respect to the sequence of norms $\|\cdot\|_m$ for each $\lambda \in I$. Additionally, if there is $d>0$ such that, for all $m \in \N$, $x \in X$ and $\lambda,\mu \in I$, we have
$$
\|\left(B_m(\lambda)-B_m(\mu)\right)x\|_{m+1} \le d\frac{\ln((m+2)/(m+1))}{\ln(m+2)} \|\lambda-\mu\| \, \|x\|_m,
$$
we can choose the projections $P_{m,\lambda}$ in such a way that $\lambda \mapsto P_{m,\lambda}$ is locally Lipschitz.
\end{corollary}

\begin{corollary}[of Theorem~\ref{teo:robustness-strong-dich-bounds}]
Let $\mu$ be a discrete growth rate and assume that $(A_m)_{m \in \N}$ admits a strong nonuniform logarithmic dichotomy.
Assume further that, for all $m \in \N$ and $\lambda,\mu \in I$, there is $c>0$ such that
$$\|B_m(\lambda)\| \le \frac{\ln((m+2)/(m+1))}{(\ln(m+2))^{1+\eps}} c,$$
where $\eps>0$ is the exponent of $\mu_n$ the definition of strong nonuniform $\mu$-dichotomy.
Then, if $c>0$ is sufficiently small, the sequence $(A_m+B_m(\lambda))_{m \in \N}$ admits a strong nonuniform $\mu$-dichotomy for each $\lambda \in I$. Additionally, if there is $d>0$ such that
$$\|\left(B_m(\lambda)-B_m(\mu)\right)\| \le d \frac{\ln((m+2)/(m+1))}{(\ln(m+2))^{1+\eps}} |\lambda-\mu|,$$
for all $m \in \N$ and $\lambda,\mu \in I$, we can choose the projections $P_{m,\lambda}$ in such a way that $\lambda \mapsto P_{m,\lambda}$ is locally Lipschitz.
\end{corollary}

\end{document}